\documentclass[11pt]{amsart}
\usepackage{amsmath,amsthm}
\usepackage{graphics}
\usepackage{latexsym}
\usepackage{verbatim}
\usepackage{amsmath}
\usepackage{amsthm}
\usepackage{amssymb}
\usepackage{epsfig}
\usepackage{epstopdf}
\usepackage{subfig}
\usepackage{hyperref}
\usepackage[]{hyperref}
\usepackage{setspace}
\usepackage[noend]{algpseudocode}
\usepackage{algorithmicx,algorithm}
\hypersetup{urlcolor=blue, citecolor=red}
\usepackage{dsfont}
\usepackage[table]{xcolor}
\usepackage{multirow}
\usepackage{float}
\usepackage{tikz}
\usetikzlibrary{arrows, calc, shapes, positioning, decorations.pathreplacing}
\usepackage{comment}
\usepackage{MnSymbol}
\usepackage{diagbox}
\usepackage{multirow}
\newcommand{\minitab}[2][l]{\begin{tabular}{#1}#2\end{tabular}}

\ifpdf
  \DeclareGraphicsExtensions{.eps,.pdf,.png,.jpg}
\else
  \DeclareGraphicsExtensions{.eps}
\fi
\newtheorem{theorem}{Theorem}[section]

\newtheorem{remark}[theorem]{Remark}

\title[Asymptotic preserving scheme for anisotropic elliptic equations with DNN]{Asymptotic preserving scheme for anisotropic elliptic equations with deep neural network}

\author[Long Li and Chang Yang]{}


\email{long.li@hit.edu.cn}
\email{yangchang@hit.edu.cn}

\begin{document}
\maketitle

\centerline{\scshape Long Li, Chang Yang}


\medskip
{\footnotesize
    \centerline{School of Mathematics, Harbin Institute of Technology,}
    \centerline{No. 92 West Dazhi Street, Nangang District, 150001 Harbin, China}
}

\begin{abstract}
In this paper, a new asymptotic preserving (AP) scheme is proposed for the anisotropic elliptic equations. Different from previous AP schemes, the actual one is based on first-order system least-squares for second-order partial differential equations, and it is uniformly well-posed with respect to anisotropic strength. The numerical computation is realized by a deep neural network (DNN), where least-squares functionals are employed as loss functions to determine parameters of DNN. Numerical results show that the current AP scheme is easy for implementation and is robust to approximate solutions or to identify anisotropic strength in various  2D and 3D tests.
\end{abstract}


\tableofcontents

\section{Introduction}
\label{sec:introduction}

The study of the anisotropic elliptic equations~\cite{Degond2017Asymptotic} is very important for two aspects: on the one hand, in mathematical sense, the  anisotropic elliptic equations become ill-posed (loss uniqueness of solution) when anisotropic strength, denoted by $\varepsilon$, goes to 0. On the other hand, in numerical sense, the numerical results are polluted once $\varepsilon\ll1$. Moreover, the  anisotropic elliptic equations have important applications in the physics of magnetized plasma, such as confinement plasma~\cite{Meiss2003Plasma}, plasma of ionosphere~\cite{Besse2004A}, plasma propulsion~\cite{Meier2011}, {\it  etc}.

An efficient method for solving the anisotropic elliptic equations is asymptotic preserving (AP) scheme~\cite{yang2020Preserving}. The basic idea for AP scheme is to identify a unique solution as $\varepsilon\to0$. Various types of AP schemes have been proposed in literature, such as the duality-based method~\cite{degond_asymptotic_2009, degond_duality-based_2012, Christophe2013Efficient, yang_iterative_2018}, the micro-macro method~\cite{2012An}, the two field iterated method~\cite{deluzet_two_2019, yang_numerical_2019}, Tang's method~\cite{tang_asymptotic_2017, wang_uniformly_2018}, {\it  etc}. In most of these methods, an auxillary variable was introduced to eliminate anisotropy of original  anisotropic elliptic equations, so we have a system of second-order elliptic equations.

In this paper, a new AP scheme based on first-order system (FOS) least-squares (LS) is proposed. The first-order system least-squares for second-order elliptic partial differential equations (PDEs) have been studied in~\cite{Cai1994First, cai2020deep}, where they pointed out two striking features: ``on the one hand, it naturally symmetrizes and stabilizes the original problem; on the other hand, value of the corresponding LS functional at the current approximation is an accurate {\it a posteriori} error estimator''. We thus develop APFOS scheme by introducing new auxillary variables, which are equal to directional gradients of solution. Another important difference of our APFOS scheme is that these auxillary variables may not be unique, particularly in the case with closed magnetic field, instead they directional derivatives are unique and used to determine solution of anisotropic elliptic equation. Therefore, our APFOS scheme is uniformly well-posed with respect to $\varepsilon\geq0$. We refer to Section~\ref{sec:APFOSscheme} for detailed explanation. In our knowledge, this is the first time to combine AP scheme with FOS of elliptic PDEs, and thus AP scheme may benefit the striking features of FOS. In contrast, the main task of the existing AP schemes in literature is to design a way to uniquely determine auxillary variables. Therefore, our APFOS scheme is free from this design difficult. Finally, we rewrite the APFOS scheme into a LS formulation, that is a minimization problem of LS functional, where boundary conditions are also added with proper scales.
Next step is to solve this APFOS-LS formulation efficiently.

Recently, solving PDEs in deep neural network (DNN) framework becomes popular~\cite{2019Machine, Raissi2019, Cme2020Deep}. Compare to classical methods for numerical resolution of PDEs, the DNN is quite different. For instance, in finite element method or spectral method, {\it  etc}., a discrete functional space is firstly introduced, which has finite dimensions and is an approximation of continuous functional space. Once a basis of the discrete functional space is defined, then unknowns of PDEs are just linear combination of the basis, and the combination coefficients are parameters to be determined. Obviously, more accurate solution means more parameters required. In contrast, in DNN framework, we define a fully nonlinear function, with certain weights and biases to be determined. Thanks to the non-linearity, the DNN solution may describe unknowns of PDEs with few parameters. Particularly, some recent works focus on solving elliptic PDEs in DNN framework~\cite{cai2020deep, HAN2020109672, 2020Int}. They have shown that, on the one hand, DNN can solve many types of elliptic PDEs, on the other hand, DNN can treat well some particular cases without special effort, which is not the case for classical methods.

Therefore, we discretize the APFOS-LS functional  in DNN framework, that is to substitute solutions of APFOS scheme by a DNN function and to discretize the APFOS-LS functional by a summation of the equations and boundary condition evaluated on certain collocation points. Minimizing the discrete APFOS-LS functional with respect to weights and biases of DNN yields an optimal DNN, which is approximate solutions of APFOS scheme. The APFOS-LS method in DNN framework is easy to implement since rich tools in deep learning can be  directly employed, such as dominant automatic differentiation technique, stochastic gradient descent methods, {\it  etc}. Numerical tests in 2D/3D show our APFOS-LS method is uniformly accurate with respect to $\varepsilon$, even for a difficult test case with closed magnetic field. In contrast, the non-AP LS method does not work for strong anisotropy. This result confirms that DNN framework works only for well-posed PDEs~\cite{Raissi2019}.

A reliable prediction also needs a good knowledge of model parameters. In most real applications, such elements are of physical sense but unavailable, e.g., variational data assimilation for weather forecast~\cite{LeDimet1986}, seismic imaging in geophysics~\cite{R2006A}, computed tomography in medical imaging~\cite{2017Ultrasonic}. For the anisotropic elliptic equation, the anisotropic strength $\varepsilon$ also has an influence on the well-posedness. In this scenario of unknown anisotropic strength $\varepsilon$, it is necessary to determine the coefficient by providing additional \emph{priori} information, which is called inverse problem. In general, a least-squares formulation is established and solved by the gradient descent algorithm. Owing to the ability of tackling strong anisotropy and efficient computation, in this work the APFOS-LS method in the DNN framework is proposed for optimizing the anisotropic strength $\varepsilon$.

This paper is organized as follows. In Section~\ref{sec:APFOS}, the anisotropic elliptic equation is stated, then we introduce APFOS schemes and their corresponding LS formulations. In Section~\ref{sec:DNNFOSAP}, we first briefly review DNN framework, then the APFOS schemes are numerically resolved based on DNN framework. In Section~\ref{sec:parameteresti}, an APFOS identification scheme, for estimating the anisotropic strength parameter $\varepsilon$, is presented. Finally, in Section~\ref{sec:numresults}, numerical implementation is explained and various numerical results are shown.

\section{An asymptotic preserving scheme for anisotropic elliptic equation based on  first order system} \label{sec:APFOS}

\subsection{Problem statement}

In this paper, the following anisotropic elliptic equation is considered:
\begin{equation}\label{eq:anisotropic_pb}
\left\{
\begin{array}{ll}
-\Delta_\bot \phi^\varepsilon - \frac{1}{\varepsilon}\Delta_{\|}\phi^{\varepsilon} = f^\varepsilon, & \text{in } \Omega,\, \\
\phi^\varepsilon = g^\varepsilon, & \text{on }\Gamma_D,\, \\
\nabla_\bot \phi^\varepsilon \cdot \mathbf{n} + \frac{1}{\varepsilon}
\nabla_{\|}\phi^\varepsilon \cdot \mathbf{n} = 0, &  \text{on }\Gamma_N,
\end{array}
\right.
\end{equation}
where $\Omega$ is an open set in $\mathbb{R}^2$ or  $\mathbb{R}^3$, its boundary is defined by $\partial \Omega= \Gamma_D\cup\Gamma_N$, $\mathbf{n}$ is exterior normal vector of $\Gamma_N$. Let $\mathbf{b}$ be a magnetic field, then we can define some parallel or perpendicular differential operators:
\begin{equation*}
\begin{array}{ll}
\nabla_{\|} \phi^\varepsilon = \mathbf{b}\otimes\mathbf{b}\nabla\phi^\varepsilon, & \nabla_{\bot} \phi^\varepsilon = (I - \mathbf{b}\otimes\mathbf{b})\nabla\phi^\varepsilon, \, \\
\Delta_\| \phi^\varepsilon = \nabla\cdot \nabla_{\|}\phi^\varepsilon, & \Delta_\bot \phi^\varepsilon = \nabla\cdot \nabla_{\bot}\phi^\varepsilon.
\end{array}
\end{equation*}
The problem~\eqref{eq:anisotropic_pb} is well-posed for any $\varepsilon>0$. Indeed, in particular, let $\Omega$ be a convex polygon, and $f^\varepsilon\in L^2(\Omega)$, $g^\varepsilon\in H^{3/2}(\Gamma_D)$, then we have a unique solution $\phi^\varepsilon\in H^2_{D}(\Omega)$  of the problem~\eqref{eq:anisotropic_pb}, with
\begin{equation*}
H^2_{D}(\Omega) = \{v\in H^2(\Omega):\, v|_{\Gamma_D} = g^\varepsilon\}.
\end{equation*}

Formally, let $\varepsilon$ go to 0, we get a degenerate problem:
\begin{equation}\label{eq:degenerate_pb}
\left\{
\begin{array}{ll}
\Delta_{\|}\phi^{0} = 0, & \text{in } \Omega,\, \\
\phi^0 = g^0, & \text{on }\Gamma_D,\, \\
\nabla_{\|}\phi^0 \cdot \mathbf{n} = 0, &  \text{on }\Gamma_N.
\end{array}
\right.
\end{equation}
The problem~\eqref{eq:degenerate_pb} is ill-posed due to loss of uniqueness of solution, since any function belongs to set $$H^2_{D,\|} = \{v\in H^2(\Omega):\, v|_{\Gamma_D} = g^0,\, \nabla_\|v=0\text{ in }\Omega\}$$ is solution of~\eqref{eq:degenerate_pb}. In another word, the solution of~\eqref{eq:degenerate_pb} is in kernel of differential operator $\nabla_\|$.

We then define a limit problem
\begin{equation}\label{eq:limit_pb}
\left\{
\begin{array}{ll}
-\Delta_{\bot}\bar\phi = f^0, & \text{in } \Omega,\, \\
\bar\phi = g^0, & \text{on }\Gamma_D,\, \\
\nabla_{\bot}\bar\phi \cdot \mathbf{n} = 0, &  \text{on }\Gamma_N,\, \\
\nabla_{\|}\bar\phi = 0, &  \text{in } \Omega.
\end{array}
\right.
\end{equation}
This limit problem~\eqref{eq:limit_pb} has unique solution $\bar\phi \in H^2_{D,\|}$ thanks to the constraint $\nabla_{\|}\bar\phi = 0  \text{ in } \Omega$.

The question now is how to find a way such that a unique solution is uniformly defined for any $\varepsilon\geq0$, which is the main subject of next subsection.

\subsection{An asymptotic preserving scheme based on first order system}\label{sec:APFOSscheme}

In the sequel, we will omit the superscript $\varepsilon$ if there is no confusion.
\paragraph{\bf Aligned 2D case} To give the basic idea, let us first consider a simple case in 2D, that is the magnetic field lines $\mathbf{b}$ aligned with $z$-axis:
\begin{equation}\label{eq:anisotropic_aligned2D}
\left\{
\begin{array}{ll}
-\partial^2_{x} \phi - \frac{1}{\varepsilon}\partial^2_{z}\phi = f, & \text{in } \Omega,\, \\
\phi = 0, & \text{on }\Gamma_D,\, \\
\partial_z\phi = 0, &  \text{on }\Gamma_N,
\end{array}
\right.
\end{equation}
where $\Omega=(0,1)\times(0,1)$, $\Gamma_D=\{0,1\}\times(0,1)$, $\Gamma_N=(0,1)\times\{0,1\}$.
The equation~\eqref{eq:anisotropic_aligned2D} is well-posed for any $\varepsilon>0$. Formally let $\varepsilon\to0$, the equation~\eqref{eq:anisotropic_aligned2D} becomes ill-posed , since any function $\phi$ depending only on $x$ and satisfying $\phi=0$ on $\Gamma_D$ is the solution of degenerate problem.

A new asymptotic preserving (AP) scheme is introduced by introducing the ansatz:
\begin{equation}
\partial_x\phi = \tau\text{ and }\partial_z\phi = \varepsilon\sigma,\quad\text{in }\Omega.
\end{equation}
Thus the new AP system reads:
\begin{subequations}\label{eqs:AP2Daligned}
\begin{equation}\label{eq:AP2Daligned}
\left\{
\begin{array}{ll}
-\partial_x\tau - \partial_z\sigma = f, & \text{in } \Omega,\, \\
\partial_x\phi = \tau,& \text{in }\Omega,\,\\
\partial_z\phi = \varepsilon\sigma,&\text{in }\Omega,
\end{array}
\right.
\end{equation}
with boundary conditions:
\begin{equation}\label{eq:AP2DalignedBC}
\left\{
\begin{array}{ll}
\phi=0, & \text{on } \Gamma_D,\, \\
\partial_z\phi = 0,& \text{on }\Gamma_N,\,\\
\sigma=0,&\text{on }\Gamma_N.
\end{array}
\right.
\end{equation}
\end{subequations}

Next theorem presents the well-posedness of the AP system~\eqref{eq:AP2Daligned}-\eqref{eq:AP2DalignedBC} (APFOS scheme):
\begin{theorem}\label{thm:APalinged}
Let
\begin{align*}
H^1_D =\{v\in H^1(\Omega):v|_{\Gamma_D}=0\},\,
H^1_N =\{v\in H^1(\Omega):v|_{\Gamma_N}=0\}.
\end{align*}
For any $\varepsilon\geq0$, $f\in L^2(\Omega)$, we have a unique triple $(\phi,\sigma,\tau)\in H^1_D\times H^1_N\times H^1(\Omega)$ being the solution of the  AP system~\eqref{eqs:AP2Daligned}.
\end{theorem}

\begin{proof}
It is evident that  the  AP system~\eqref{eqs:AP2Daligned} is well-posed for any $\varepsilon>0$. Indeed,  the  AP system~\eqref{eqs:AP2Daligned} is equivalent to the equation~\eqref{eq:anisotropic_aligned2D} for $\varepsilon>0$.

Now we will focus on the case $\varepsilon=0$, where  the  system~\eqref{eq:AP2Daligned} becomes
\begin{equation}\label{eq:AP2Daligned_limit}
\left\{
\begin{array}{ll}
-\partial_x\tau - \partial_z\sigma = f, & \text{in } \Omega,\, \\
\partial_x\phi = \tau,& \text{in }\Omega,\,\\
\partial_z\phi =0,&\text{in }\Omega.
\end{array}
\right.
\end{equation}
The constraint $\partial_z\phi=0$ in $\Omega$ of~\eqref{eq:AP2Daligned_limit} implies $\phi$ only depending on $x$, {\it i.e.} $\phi(x,z)\equiv\phi(x)$, which also implies $\tau(x,z)\equiv\tau(x)$ thanks to $\partial_x\phi = \tau$.

Let $\bar{\Omega}=(0,1)$. Integrating the system~\eqref{eq:AP2Daligned_limit} on $z$-axis, we have
\begin{equation}\label{eq:AP2Daligned_mean}
\left\{
\begin{array}{ll}
-\partial_x\bar\tau = \bar{f}, & \text{in } \bar{\Omega},\, \\
\partial_x\bar\phi =\bar \tau,& \text{in }\bar{\Omega},\,\\
\bar\phi=0, & \text{on } \{0,1\},
\end{array}
\right.
\end{equation}
where the symbol bar means average value along $z$-axis of a function, {\it e.g.}
\begin{align*}
\bar{f}(x)= \int_0^1 f(x,z)dz.
\end{align*}
The system~\eqref{eq:AP2Daligned_mean} admits a unique couple of solutions $(\bar\phi,\bar\tau)\in \{v\in H^1(\bar{\Omega}):v(0)=v(1)=0\}\times H^1(\bar{\Omega})$~\cite{Evans2010}.

Finally, we reformulate  the first equation of~\eqref{eq:AP2Daligned_limit} as
\begin{align*}
\partial_z\sigma = -\partial_x \tau - f,
\end{align*}
which gives $\sigma$, thanks to boundary condition of $\sigma$,  by integration, {\it i.e.}
\begin{align*}
\sigma(x,z)= - z\,\partial_x\tau(x) - \int^z_0 f(x,s)ds.
\end{align*}
\end{proof}

\begin{remark}
If the boundary conditions at the end of $z$-axis of~\eqref{eq:AP2DalignedBC} is replaced by periodic boundary conditions, {\it i.e.}
\begin{equation}\label{eq:AP2DalignedBCperiodic}
\left\{
\begin{array}{ll}
\phi(0,z)=\phi(1,z)=0, & z\in\bar{\Omega},\, \\
\partial_z\phi(x,0) = \partial_z\phi(x,1),& x\in\bar{\Omega},\,\\
\sigma(x,0)=\sigma(x,1),& x\in\bar{\Omega},
\end{array}
\right.
\end{equation}
then the couple $(\phi,\tau)$ in the APFOS scheme remains well-defined, but $\sigma$ is not unique if $\varepsilon=0$. Indeed, we have
 \begin{align*}
\sigma(x,z)= - z\,\partial_x\tau(x) - \int^z_0 f(x,s)ds + c(x),
\end{align*}
for any $c(x)$ depending only on $x$. However, the partial derivative $\partial_z\sigma$ is unique since $\partial_z\sigma = -\partial_x \tau - f$.

\end{remark}

\paragraph{\bf General 2D case}
Let us consider now general 2D anisotropic elliptic equation:
\begin{equation}\label{eq:anisotropic_nonaligned2D}
\left\{
\begin{array}{ll}
-\Delta_\bot \phi - \frac{1}{\varepsilon}\Delta_{\|}\phi = f, & \text{in } \Omega,\, \\
\phi = g, & \text{on }\Gamma_D,\, \\
\nabla_\bot \phi \cdot \mathbf{n} + \frac{1}{\varepsilon}
\nabla_{\|}\phi \cdot \mathbf{n} = 0, &  \text{on }\Gamma_N.
\end{array}
\right.
\end{equation}
The magnetic field and its orthogonal vector are defined by
\begin{align*}
\mathbf{b} = (b_x, b_z)^T,\quad \mathbf{b}^{\bot}=(-b_z,b_x)^T.
\end{align*}
Similar to the aligned 2D case, the APFOS scheme for general 2D case is given by
\begin{subequations}\label{eqs:AP2Dnonaligned}
\begin{equation}\label{eq:AP2Dnonaligned}
\left\{
\begin{array}{ll}
-\nabla\cdot(\tau\mathbf{b}^\bot) - \nabla\cdot(\sigma\mathbf{b}) = f, & \text{in } \Omega,\, \\
\nabla_\bot\phi:=\mathbf{b}^\bot\otimes\mathbf{b}^{\bot}\nabla\phi = \tau\mathbf{b}^{\bot},& \text{in }\Omega,\,\\
\nabla_\|\phi := \mathbf{b}\otimes\mathbf{b}\nabla\phi= \varepsilon\sigma\mathbf{b},&\text{in }\Omega,
\end{array}
\right.
\end{equation}
with boundary conditions:
\begin{equation}\label{eq:AP2DnonalignedBC}
\left\{
\begin{array}{ll}
\phi=g, & \text{on } \Gamma_D,\, \\
\nabla_\bot \phi \cdot \mathbf{n} + \frac{1}{\varepsilon}
\nabla_{\|}\phi \cdot \mathbf{n} = 0,& \text{on }\Gamma_N,\,\\
\tau\mathbf{b}^\bot \cdot \mathbf{n} +
\sigma\mathbf{b} \cdot \mathbf{n} = 0,&\text{on }\Gamma_N.
\end{array}
\right.
\end{equation}
\end{subequations}
It is clear that, by taking $\mathbf{b}=(0,1)^T$, $\Omega=(0,1)\times(0,1)$,  $\Gamma_D=\{0,1\}\times(0,1)$, $\Gamma_N=(0,1)\times\{0,1\}$, the APFOS scheme~\eqref{eqs:AP2Dnonaligned} in general 2D case is nothing but the APFOS scheme~\eqref{eqs:AP2Daligned}.

The well-posedness of the APFOS scheme~\eqref{eqs:AP2Dnonaligned} in general 2D case is summarized in next theorem:
\begin{theorem}\label{thm:WP_AP_2D}
Suppose that the ends of magnetic field lines only cross on $\Gamma_N$ and  the boundary $\Gamma_D$ coincides with the magnetic field lines. Let
\begin{align*}
H^1_{D,g} = \{v\in H^1(\Omega):\, v|_{\Gamma_D} = g\}.
\end{align*}
Then for any $\varepsilon\geq0$, $f\in L^2(\Omega)$, we have a unique couple $(\phi,\sigma,\tau)\in H^1_{D,g}\times H^1(\Omega)\times H^1(\Omega)$ being the solution of the  APFOS scheme~\eqref{eqs:AP2Dnonaligned}.
\end{theorem}
\begin{proof}
We will only consider the case  $\varepsilon=0$, whereas the well-posedness for the case $\varepsilon>0$ is evident.

Let us make a change of variable as
\begin{align*}
x = x(X,Z), \quad z = z(X,Z),
\end{align*}
with scale factors $h_X(X,Z)$ and $h_Z(X,Z)$ respectively.
The $Z$-axis is aligned with the magnetic field, while $X$-axis is perpendicular with  the magnetic field, {\it i.e.}
\begin{align*}
\hat{X}=-b_z \hat{x} + b_x\hat{z},\,\hat{Z}=b_x \hat{x} + b_z\hat{z},
\end{align*}
where $\hat{X}$ and $\hat{Z}$ are unit basis in new curvilinear coordinates.
 For simplicity, in the sequel we suppose that $(X,Z)\in[X_0,X_{\max}]\times[Z_0,Z_{\max}]$. However, the general case can be proved similarly.

Thanks to the  change of variable,  the system~\eqref{eq:AP2Dnonaligned} become
\begin{equation}\label{eq:AP2Dnonaligned_CV}
\left\{
\begin{array}{ll}
-\partial_X(h_Z\tau) - \partial_Z(h_X\sigma) = h_X\,h_Z\,f, & \text{in } \Omega,\, \\
\partial_X\phi=h_X\tau=\frac{h_X}{h_Z}(h_Z\tau),& \text{in }\Omega,\,\\
\partial_Z\phi=\varepsilon h_Z\sigma=0,&\text{in }\Omega.
\end{array}
\right.
\end{equation}
Let us denote
\begin{align*}
T=h_Z\tau,\,\Sigma=h_X\sigma,\,F=h_X\,h_Z\,f.
\end{align*}
so the system~\eqref{eq:AP2Dnonaligned_CV} becomes
\begin{equation}\label{eq:AP2Dnonaligned_CV2}
\left\{
\begin{array}{ll}
-\partial_X T - \partial_Z\Sigma =F, & \text{in } \Omega,\, \\
\partial_X\phi=\frac{h_X}{h_Z}T,& \text{in }\Omega,\,\\
\partial_Z\phi=0,&\text{in }\Omega.
\end{array}
\right.
\end{equation}
Thanks to the third equation of~\eqref{eq:AP2Dnonaligned_CV2}, we have
\begin{align*}
\phi(X,Z)\equiv\phi(X).
\end{align*}
Integrating the first equation of~\eqref{eq:AP2Dnonaligned_CV2} along magnetic field lines $\mathbf{b}$ ($Z$-axis) yields
\begin{align*}
-\partial_X \bar{T}(X) - (\Sigma(X,Z_{\max}) - \Sigma(X,Z_0)) = \bar{F}(X),
\end{align*}
where the symbol bar means average value along $Z$-axis of a function, {\it e.g.}
\begin{align*}
\bar{F}(X) = \frac{1}{Z_{\max} - Z_0}\int_{Z_0}^{Z_{\max}}F(X,Z)dZ.
\end{align*}

Thanks to the third equation of~\eqref{eq:AP2DnonalignedBC}, we have
\begin{align*}
\sigma(X,Z_0) =- \bar{T} (X) \left(\frac{\mathbf{b}^\bot\cdot\mathbf{n}}{\mathbf{b}\cdot\mathbf{n}} \frac{h_X}{h_Z}\right)(X,Z_0),\quad
\sigma(X,Z_{\max}) = - \bar{T}(X)\left( \frac{\mathbf{b}^\bot\cdot\mathbf{n}}{\mathbf{b}\cdot\mathbf{n}}\frac{h_X}{h_Z}\right)(X,Z_{\max}).
\end{align*}
Moreover, thanks to the assumption on the magnetic field,  we have $g=\bar{g}$ on $ \{X_0,X_{\max}\}$. Thus
\begin{equation}\label{eq:AP2Dnonaligned_mean2}
\left\{
\begin{array}{ll}
-\partial_X\bar{T} + C(X)\bar{T}= \bar{f}, & \text{in } \bar{\Omega},\, \\
\bar{T} =\overline{\frac{h_Z}{h_X}}\partial\bar{\phi},& \text{in }\bar{\Omega},\,\\
\bar\phi=\bar{g}, & \text{on } \{X_0,X_{\max}\},
\end{array}
\right.
\end{equation}
where $\bar{\Omega}=[X_0,X_{\max}]$, $C(X) = \left( \frac{\mathbf{b}^\bot\cdot\mathbf{n}}{\mathbf{b}\cdot\mathbf{n}}\frac{h_X}{h_Z}\right)(X,Z_{\max}) -
\left(\frac{\mathbf{b}^\bot\cdot\mathbf{n}}{\mathbf{b}\cdot\mathbf{n}} \frac{h_X}{h_Z}\right)(X,Z_0).
$
The system~\eqref{eq:AP2Dnonaligned_mean2}  permits unique solution provided $C(X)$ is bounded in $[X_0,X_{\max}]$~\cite{Evans2010}, which is true thanks to the assumption of theorem.
Therefore, we have a unique couple $(\phi,\tau)=(\bar\phi,\frac{1}{h_X}\partial_X\bar{\phi})\in H^1_{D,g}\times H^1(\Omega)$.

Finally, following the last lines of the proof of Theorem~\ref{thm:APalinged}, we deduce the uniqueness of $\sigma$, thus we achieve the proof.
\end{proof}

We remark that, on the one hand, for the case with closed magnetic field, that is periodic boundary in $Z$-axis, we have
\begin{align*}
\Sigma(X,Z_0) = \Sigma(X,Z_{\max}).
\end{align*}
Then, following the same arguments in the proof of Theorem~\ref{thm:WP_AP_2D}, we have again a unique couple $(\phi,\tau)\in H^1_{D,g}\times H^1(\Omega)$.
On the other hand, for the same reason as that in the aligned case, we loss the uniqueness of the variable $ \sigma$  for the case $\varepsilon=0$ with closed magnetic field.
\paragraph{\bf General 3D case}
The general 3D  anisotropic elliptic equation takes the same form as in~\eqref{eq:anisotropic_nonaligned2D}. However, the magnetic field and its orthogonal vectors are defined differently
\begin{align*}
\mathbf{b}=(b_x,b_y,b_z)^T,\quad \mathbf{b}^{\bot,1} = \frac{ \mathbf{B}^{\bot,1}}{|\mathbf{B}^{\bot,1}|},\quad  \mathbf{b}^{\bot,2} = \frac{ \mathbf{B}^{\bot,2}}{|\mathbf{B}^{\bot,2}|},
\end{align*}
with
\begin{align*}
\mathbf{B}^{\bot,1} = \mathbf{b} \times (0,0,1)^T,\quad \mathbf{B}^{\bot,2} = \mathbf{b}\times\mathbf{B}^{\bot,1}.
\end{align*}
Therefore, we need introduce correspondingly three auxiliary variables.
The APFOS scheme for general 3D case is given by
\begin{subequations}\label{eqs:AP3Dnonaligned}
\begin{equation}\label{eq:AP3Dnonaligned}
\left\{
\begin{array}{ll}
-\nabla\cdot(\tau\mathbf{b}^{\bot,1} + \chi\mathbf{b}^{\bot,2}) - \nabla\cdot(\sigma\mathbf{b}) = f, & \text{in } \Omega,\, \\
\nabla\phi\cdot\mathbf{b}^{\bot,1} = \tau,& \text{in }\Omega,\,\\
\nabla\phi\cdot\mathbf{b}^{\bot,2} = \chi,& \text{in }\Omega,\,\\
\nabla\phi\cdot\mathbf{b}= \varepsilon\sigma,&\text{in }\Omega,
\end{array}
\right.
\end{equation}
with boundary conditions:
\begin{equation}\label{eq:AP3DnonalignedBC}
\left\{
\begin{array}{ll}
\phi=g, & \text{on } \Gamma_D,\, \\
\nabla_\bot \phi \cdot \mathbf{n} + \frac{1}{\varepsilon}
\nabla_{\|}\phi \cdot \mathbf{n} = 0,& \text{on }\Gamma_N,\,\\
(\tau\mathbf{b}^{\bot,1} + \chi\mathbf{b}^{\bot,2}) \cdot \mathbf{n} +
\sigma\mathbf{b} \cdot \mathbf{n} = 0,&\text{on }\Gamma_N.
\end{array}
\right.
\end{equation}
\end{subequations}

By following the same arguments shown in general 2D case, we could deduce that the  APFOS scheme~\eqref{eqs:AP3Dnonaligned} in general 3D case is also well-posed.

\subsection{APFOS least-squares formulations}\label{sec:continuous_LS}
This part is devoted to introducing APFOS least-squares formulations for APFOS scheme~\eqref{eqs:AP2Dnonaligned} and~\eqref{eqs:AP3Dnonaligned}.

For 2D case, the APFOS-LS formulation is to find $(\phi,\tau,\sigma)\in (H^1(\Omega))^3$ such that
\begin{subequations}\label{eqs:continuous_APFOSLSFormulation_anisotropic}
\begin{align}\label{eq:continuous_APFOSLSFormulation_anisotropic}
\mathcal{G}(\phi,\tau,\sigma ;\mathbf{f}) = \min_{(\psi,\xi,\zeta)\in (H^1(\Omega))^3}\mathcal{G}(\psi,\xi,\zeta ;\mathbf{f}),
\end{align}
where $\mathbf{f}= (f,g)$,
and the APFOS-LS functional with proper scales is
\begin{equation}\label{eq:continuous_APFOSLSF_anisotropic}
\begin{array}{ll}
\mathcal{G}(\psi,\xi,\zeta ;\mathbf{f}) = & \|\nabla\cdot(\xi\mathbf{b}^\bot) + \nabla\cdot(\zeta\mathbf{b}) + f\|^2_{0,\Omega} + \|\nabla\psi\cdot\mathbf{b}^\bot - \xi\|^2_{0,\Omega} +  \|\nabla\psi\cdot\mathbf{b} - \varepsilon\zeta\|^2_{0,\Omega} \\[3mm]
& + \|\psi- g\|^2_{1/2,\Gamma_D} + \|\varepsilon\nabla_\bot \psi \cdot \mathbf{n} +
\nabla_{\|}\psi \cdot \mathbf{n}\|^2_{-1/2,\Gamma_N}
+ \|\xi\mathbf{b}^\bot \cdot \mathbf{n} +
\zeta\mathbf{b} \cdot \mathbf{n}\|^2_{-1/2,\Gamma_N},
\end{array}
\end{equation}
\end{subequations}
for all $(\psi,\xi,\zeta)\in (H^1(\Omega))^3$.
There exists $(\phi,\tau,\sigma)\in (H^1(\Omega))^3$, which can minimize APFOS-LS functional \eqref{eq:continuous_APFOSLSF_anisotropic}. Moreover, $(\phi,\tau)$ is also the unique solution to the APFOS scheme~\eqref{eqs:AP2Dnonaligned} such that $(\phi,\tau)\in H^1_{D,g}\times H^1(\Omega)$.

Similarly, for 3D APFOS scheme~\eqref{eqs:AP3Dnonaligned},
 the APFOS-LS formulation is to find $(\phi,\tau,\chi,\sigma)\in (H^1(\Omega))^4$ such that
 \begin{subequations}
\begin{align}
\mathcal{G}(\phi,\tau,\chi,\sigma ;\mathbf{f}) = \min_{(\psi,\xi,\kappa,\zeta)\in (H^1(\Omega))^4}\mathcal{G}(\psi,\xi,\kappa,\zeta ;\mathbf{f}).
\end{align}
where the APFOS-LS functional is as follows
\begin{equation}
\begin{array}{ll}
\mathcal{G}(\psi,\xi,\kappa,\zeta ;\mathbf{f}) = & \|\nabla\cdot(\xi\mathbf{b}^{\bot,1} + \kappa\mathbf{b}^{\bot,2} ) + \nabla\cdot(\zeta\mathbf{b}) + f\|^2_{0,\Omega} + \|\nabla\psi\cdot\mathbf{b}^{\bot,1} - \xi\|^2_{0,\Omega}\\[3mm]
& +   \|\nabla\psi\cdot\mathbf{b}^{\bot,2} - \kappa\|^2_{0,\Omega} + \|\nabla\psi\cdot\mathbf{b} - \varepsilon\zeta\|^2_{0,\Omega} + \|\psi- g\|^2_{1/2,\Gamma_D}
\\[3mm]
& + \|\varepsilon\nabla_\bot \psi \cdot \mathbf{n} +
\nabla_{\|}\psi \cdot \mathbf{n}\|^2_{-1/2,\Gamma_N} + \|\xi\mathbf{b}^{\bot,1} \cdot \mathbf{n} +\kappa\mathbf{b}^{\bot,2} \cdot \mathbf{n} +
\zeta\mathbf{b} \cdot \mathbf{n}\|^2_{-1/2,\Gamma_N},
\end{array}
\end{equation}
\end{subequations}
for all $(\psi,\xi,\kappa,\zeta)\in (H^1(\Omega))^4$.

\section{Deep neural network for asymptotic preserving scheme based on first order system} \label{sec:DNNFOSAP}

\subsection{Deep neural network architecture}
We will briefly describe the deep neural network (DNN) structure used in this paper. For a more detailed review of DNN structure, we refer to~\cite{cai2020deep} and references therein.

Simply to say, a deep neural network defines a nonlinear function
\begin{align*}
\mathcal{N}:\mathbf{x}\in\mathbb{R}^d\to \mathbf{y}=\mathcal{N}(\mathbf{x})\in\mathbb{R}^c,
\end{align*}
where $d$ and $c$ are dimensions of input $\mathbf{x}$ and output $\mathbf{y}$, respectively. $\mathcal{N}(\mathbf{x})$ is a composite function of many different layers of functions as
\begin{align*}
\mathcal{N}(\mathbf{x}) =\mathcal{N}^{(L)}\circ\cdots\circ\mathcal{N}^{(2)}\circ\mathcal{N}^{(1)}(\mathbf{x}),
\end{align*}
where $L$ is the number of layers. The layer function $\mathcal{N}^{(l)}:\mathbb{R}^{n_{l-1}}\to\mathbb{R}^{n_l}$ is defined by
\begin{align*}
\mathcal{N}^{(l)}(\mathbf{x}^{(l-1)})=\varphi^{(l)}(\mathcal{W}^{(l)}\mathbf{x}^{(l-1)}+\mathcal{B}^{(l)}),\quad l=1,\dots,L-1,
\end{align*}
where $\mathbf{x}^{(l-1)}\in\mathbb{R}^{n_{l-1}}$ is the output of layer $l-1$, it is also the input of layer $l$.
$\mathcal{W}^{(l)}\in\mathbb{R}^{n_l\times n_{l-1}}$, $\mathcal{B}^{(l)}\in\mathbb{R}^{n_l}$ are called network weights and biases respectively. The application $\varphi^{(l)}:\mathbb{R}\to\mathbb{R}$ is an activation function, which is the soul of DNN, it enables the multi-layer network structure to learn complex things, complex data, and to represent the nonlinear and complex arbitrary function mapping between input and output. At the same time, the activation function can also perform data normalization to map the input data to a certain range, on the one hand, it prevents the risk of data overflow; on the other hand, since each input data of the network often has different physical meanings and different dimensions, data normalization makes all components change within a common range, so that the input components are given equal importance at the beginning of network training. In this paper, the hyperbolic tangent function is taken as the activation function
\begin{align*}
\varphi(x)=\tanh(x),\,x\in\mathbb{R}.
\end{align*}
The function $\tanh$ is very smooth, thus it will not infect the regularity of solution. Notice that the image  of $\tanh$ is $[-1,1]$, it is better to rescale the input of the first layer to $[-1,1]^d$. Finally, for the last layer $L$, there is no need to use activation function, hence
\begin{align*}
\mathcal{N}^{(L)}(\mathbf{x}^{(L-1)}) = \mathcal{W}^{(L)}\mathbf{x}^{(L-1)}+\mathcal{B}^{(L)}.
\end{align*}
It is clear that the DNN function $\mathcal{N}$ is uniquely determined by weights and biases
\begin{align*}
\mathcal{W}^{(l)},\, \mathcal{B}^{(l)},\, l=1,\dots,L,
\end{align*}
which have totally $N=\sum\limits_{l=1}^L n_l\times(n_{l-1}+1)$ parameters. Particularly, for anisotropic problem, we have $n_L=1$; for APFOS scheme in 2D case, $n_L=3$; for APFOS scheme in 3D case, $n_L=4$.
The DNN function $\mathcal{N}$  will be regarded as an approximation for resolution of the least-squares formulations in next part.

\subsection{Deep neural network for APFOS least-squares formulations}\label{sec:discrete_LS}
The idea of deep neural network for least-squares (LS) formulations is to substitute the continuous  LS formulations by their discrete approximations, where the unknowns are DNN function  $\mathcal{N}$. Then minimizing the discrete  LS formulations  gives the optimal  weights and biases.

Let us consider 2D APFOS-LS formulation, in which the approximations $\hat{\phi}(\mathbf{x},\Theta),\hat{\tau}(\mathbf{x},\Theta),\hat{\sigma}(\mathbf{x},\Theta)$ are also determined by the location $\mathbf{x}\in\mathbb{R}^2$ and  the  weights and biases  denoted as $\Theta\in\mathbb{R}^N$. The discrete APFOS-LS formulation reads
\begin{subequations}\label{eqs:discrete_LSF_AP2D}
\begin{align}
\hat{\mathcal{G}}[\hat{\phi},\hat{\tau},\hat{\sigma};\mathbf{f}](\Theta) = \min_{\tilde{\Theta}\in\mathbb{R}^N}\hat{\mathcal{G}}[\hat{\psi},\hat{\xi},\hat{\zeta} ;\mathbf{f}](\tilde{\Theta}),
\end{align}
and the following discrete APFOS-LS functional is employed
\begin{equation}\label{eq:discrete_LSF_AP2D}
\begin{array}{ll}
&\hat{\mathcal{G}}[\hat{\psi},\hat{\xi},\hat{\zeta};\mathbf{f}](\tilde{\Theta}) \\[3mm]
= &\frac{1}{N_f}\sum\limits_{i=1}^{N_f}\left(\nabla\cdot(\hat{\xi}(\mathbf{x}_i,\tilde{\Theta})\mathbf{b}^\bot(\mathbf{x}_i)) + \nabla\cdot(\hat{\zeta}(\mathbf{x}_i,\tilde{\Theta})\mathbf{b}(\mathbf{x}_i)) + f(\mathbf{x}_i)\right)^2 \\[3mm]
&+ \frac{1}{N_f}\sum\limits_{i=1}^{N_f}\left(\left(\nabla\hat{\psi}(\mathbf{x}_i,\tilde{\Theta})\cdot\mathbf{b}^\bot(\mathbf{x}_i) - \hat{\xi}(\mathbf{x}_i,\tilde{\Theta})\right)^2 + \left(\nabla\hat{\psi}(\mathbf{x}_i,\tilde{\Theta})\cdot\mathbf{b}(\mathbf{x}_i) - \varepsilon\hat{\zeta}(\mathbf{x}_i,\tilde{\Theta})\right)^2\right) \\[3mm]
& +\frac{\beta_D}{N_d}\sum\limits_{j=1}^{N_d}\left(\hat{\psi}(\mathbf{x}_j,\tilde{\Theta})- g(\mathbf{x}_j)\right)^2 +\frac{\beta_N}{N_n}\sum\limits_{k=1}^{N_n}\left(\varepsilon\nabla_\bot \hat{\psi}(\mathbf{x}_k,\tilde{\Theta}) \cdot \mathbf{n}(\mathbf{x}_k) +
\nabla_{\|}\hat{\psi}(\mathbf{x}_k,\tilde{\Theta}) \cdot \mathbf{n}(\mathbf{x}_k)\right)^2 \\[3mm]
&+\frac{\beta_N}{N_n}\sum\limits_{k=1}^{N_n}\left(\hat{\xi}(\mathbf{x}_i,\tilde{\Theta})\mathbf{b}^\bot(\mathbf{x}_k) \cdot \mathbf{n}(\mathbf{x}_k)  +
\hat{\zeta}(\mathbf{x}_k,\tilde{\Theta})\mathbf{b}(\mathbf{x}_k) \cdot \mathbf{n}(\mathbf{x}_k)\right)^2.
\end{array}
\end{equation}
\end{subequations}
where $N_f$, $N_d$, $N_n$ are numbers of collocation points in $\Omega$, $\Gamma_D$, $\Gamma_N$ respectively. The collocation points can be uniformly sampled in the domain or in some random way, which will be specified in Section~\ref{sec:numresults}.
 The norms on the boundaries are approximated by weighted $L^2$ norms, and
 $\beta_D$, $\beta_N$ represent weights in the discrete APFOS-LS functionals.
 Now minimizing~\eqref{eq:discrete_LSF_AP2D} with respect to $\Theta$ yields an optimal DNN structure solution to the APFOS scheme~\eqref{eqs:AP2Dnonaligned}.

Similarly, for 3D case, the discrete APFOS-LS formulation reads
\begin{subequations}\label{eqs:discreteAPFOS-LSfunctional}
\begin{align}
\hat{\mathcal{G}}[\hat{\phi},\hat{\tau},\hat{\chi},\hat{\sigma};\mathbf{f}]({\Theta}) = \min_{\tilde{\Theta}\in\mathbb{R}^N}\hat{\mathcal{G}}[\hat{\psi},\hat{\xi},\hat{\kappa},\hat{\zeta} ;\mathbf{f}](\tilde{\Theta}),
\end{align}
where the discrete APFOS-LS functional is
\begin{equation}\label{eq:discreteAPFOS-LSfunctional}
\begin{array}{ll}
&\hat{\mathcal{G}}[\hat{\psi},\hat{\xi},\hat{\kappa},\hat{\zeta};\mathbf{f}](\tilde{\Theta}) \\[3mm]
= &\frac{1}{N_f}\sum\limits_{i=1}^{N_f}\left(\nabla\cdot(\hat{\xi}(\mathbf{x}_i,\tilde{\Theta})\mathbf{b}^{\bot,1}(\mathbf{x}_i)) +\nabla\cdot(\hat{\kappa}(\mathbf{x}_i,\tilde{\Theta})\mathbf{b}^{\bot,2}(\mathbf{x}_i)) + \nabla\cdot(\hat{\zeta}(\mathbf{x}_i,\tilde{\Theta})\mathbf{b}(\mathbf{x}_i)) + f(\mathbf{x}_i)\right)^2 \\[3mm]
&+ \frac{1}{N_f}\sum\limits_{i=1}^{N_f}\left(\left(\nabla\hat{\psi}(\mathbf{x}_i,\tilde{\Theta})\cdot\mathbf{b}^{\bot,1}(\mathbf{x}_i) - \hat{\xi}(\mathbf{x}_i,\tilde{\Theta})\right)^2 + \left(\nabla\hat{\psi}(\mathbf{x}_i,\tilde{\Theta})\cdot\mathbf{b}^{\bot,2}(\mathbf{x}_i) - \hat{\kappa}(\mathbf{x}_i,\tilde{\Theta})\right)^2\right)\\[3mm]
& + \frac{1}{N_f}\sum\limits_{i=1}^{N_f} \left(\nabla\hat{\psi}(\mathbf{x}_i,\tilde{\Theta})\cdot\mathbf{b}(\mathbf{x}_i) - \varepsilon\hat{\zeta}(\mathbf{x}_i,\tilde{\Theta})\right)^2 \\[3mm]
& +\frac{\beta_D}{N_d}\sum\limits_{j=1}^{N_d}\left(\hat{\psi}(\mathbf{x}_j,\tilde{\Theta})- g(\mathbf{x}_j)\right)^2 +\frac{\beta_N}{N_n}\sum\limits_{k=1}^{N_n}\left(\varepsilon\nabla_\bot \hat{\psi}(\mathbf{x}_k,\tilde{\Theta}) \cdot \mathbf{n}(\mathbf{x}_k) +
\nabla_{\|}\hat{\psi}(\mathbf{x}_k,\tilde{\Theta}) \cdot \mathbf{n}(\mathbf{x}_k)\right)^2 \\[3mm]
&+\frac{\beta_N}{N_n}\sum\limits_{k=1}^{N_n}\left(\hat{\xi}(\mathbf{x}_i,\tilde{\Theta})\mathbf{b}^{\bot,1}(\mathbf{x}_k) \cdot \mathbf{n}(\mathbf{x}_k)  + \hat{\kappa}(\mathbf{x}_i,\tilde{\Theta})\mathbf{b}^{\bot,2}(\mathbf{x}_k) \cdot \mathbf{n}(\mathbf{x}_k)  +
\hat{\zeta}(\mathbf{x}_k,\tilde{\Theta})\mathbf{b}(\mathbf{x}_k) \cdot \mathbf{n}(\mathbf{x}_k)\right)^2.
\end{array}
\end{equation}
\end{subequations}

\section{Data-driven discovery in anisotropic elliptic equation} \label{sec:parameteresti}
As a direct application of APFOS scheme, in this part we focus on data-driven discovery in anisotropic elliptic equation. More specifically, we are interested  in identification of  the anisotropy strength $\varepsilon$. Only 2D cases are considered here, the ones for 3D can be extended similarly.

Let us point out that data-driven discovery can be regarded as inverse problem, that means to use some \emph{priori} information of solutions, called observation, to estimate parameters of equations. The way of estimating parameters can be made by adding those \emph{priori} information of solutions to LS functionals introduced in Section~\ref{sec:continuous_LS} and Section~\ref{sec:discrete_LS}, then we can find optimal parameter estimations by solving the new LS formulations.
\subsection{Identification of the anisotropy strength $\varepsilon$}

Suppose that $\phi_e\in H^2_D(\Omega)$ is the solution of the anisotropic equation~\eqref{eq:anisotropic_nonaligned2D}. For 2D APFOS scheme~\eqref{eqs:AP2Dnonaligned},
the identification APFOS-LS formulation is to find $\varepsilon\in \mathbb{R}^+$ and $(\phi,\tau,\sigma)\in (H^1(\Omega))^3$ such that
\begin{subequations}
\begin{align}
\mathcal{G}(\varepsilon,\phi,\tau,\sigma ;\mathbf{f}) = \min_{\tilde{\varepsilon}\in \mathbb{R}^+,(\psi,\xi,\zeta)\in (H^1(\Omega))^3}\mathcal{G}(\tilde{\varepsilon},\psi,\xi,\zeta ;\mathbf{f}),
\end{align}
where  $\mathbf{f}= (f,g)$, the APFOS-LS functional is as follows
\begin{equation}\label{eq:iden_APFOS-LSfct}
\begin{array}{ll}
\mathcal{G}(\tilde{\varepsilon},\psi,\xi,\zeta ;\mathbf{f}) = & \|\psi - \phi_e\|^2_{0,\Omega} \\[3mm]
& + \|\nabla\cdot(\xi\mathbf{b}^\bot) + \nabla\cdot(\zeta\mathbf{b}) + f\|^2_{0,\Omega} + \|\nabla\psi\cdot\mathbf{b}^\bot - \xi\|^2_{0,\Omega} +  \|\nabla\psi\cdot\mathbf{b} - \tilde{\varepsilon}\zeta\|^2_{0,\Omega} \\[3mm]
& +  \|\tilde{\varepsilon}\nabla_\bot \psi \cdot \mathbf{n} +
\nabla_{\|}\psi \cdot \mathbf{n}\|^2_{-1/2,\Gamma_N}
+ \|\xi\mathbf{b}^\bot \cdot \mathbf{n} +
\zeta\mathbf{b} \cdot \mathbf{n}\|^2_{-1/2,\Gamma_N},
\end{array}
\end{equation}
\end{subequations}
for all $\tilde{\varepsilon}\in\mathbb{R}^+$ and $(\psi,\xi,\zeta)\in (H^1(\Omega))^3$.

Observing from~\eqref{eq:iden_APFOS-LSfct}, we can find two main differences from~\eqref{eq:continuous_APFOSLSF_anisotropic}. The first difference is that the Dirichlet boundary condition constraint does not exist in~\eqref{eq:iden_APFOS-LSfct}, instead we add the observation constraint $ \|\psi - \phi_e\|^2_{0,\Omega}$. The second one is that output of optimization problem consists of both a DNN and a parameter $\varepsilon$.

\subsection{Numerical discretization}

In this part, we will propose discrete functionals to  the identification APFOS-LS functional~\eqref{eq:iden_APFOS-LSfct}, which can also be made by adding \emph{priori} information to~\eqref{eq:discrete_LSF_AP2D}. Different from continuous case, there are two aspects should be taken into account: on the one hand, the solution $\phi_e$ is only sampled on certain points, and they can be noise free or with noise; on the other hand, the parameter $\varepsilon$ should be guaranteed  always non-negative.

Following the same principle, the corresponding identification APFOS-LS formulation reads
\begin{subequations}\label{eqs:ident_discrete_LSF_AP2D}
\begin{align}\label{eq:ident_discrete_LSFormulation_APFOS2D}
\hat{\mathcal{G}}[\hat{\phi},\hat{\tau},\hat{\sigma};\mathbf{f}](\varepsilon^*,\Theta) = \min_{\tilde{\varepsilon}^*\in\mathbb{R},\tilde{\Theta}\in\mathbb{R}^N}\hat{\mathcal{G}}[\hat{\psi},\hat{\xi},\hat{\zeta} ;\mathbf{f}](\tilde{\varepsilon}^*,\tilde{\Theta}),
\end{align}
where the  corresponding discrete  functional to~\eqref{eq:iden_APFOS-LSfct} is
\begin{equation}\label{eq:ident_discrete_LSF_AP2D}
\begin{array}{ll}
&\hat{\mathcal{G}}[\hat{\psi},\hat{\xi},\hat{\zeta} ;\mathbf{f}](\tilde{\varepsilon}^*,\tilde{\Theta}) \\[3mm]
= & \frac{\beta_{e}}{N_f}\sum\limits_{i=1}^{N_f}\left(\hat{\psi}(\mathbf{x}_i,\tilde{\Theta}) - \phi_e(\mathbf{x}_i)\right)^2 \\[3mm]
%
%
&+ \frac{\beta_{f_{1}}}{N_f}\sum\limits_{i=1}^{N_f}\left(\nabla\cdot(\hat{\xi}(\mathbf{x}_i,\tilde{\Theta})\mathbf{b}^\bot(\mathbf{x}_i)) + \nabla\cdot(\hat{\zeta}(\mathbf{x}_i,\tilde{\Theta})\mathbf{b}(\mathbf{x}_i)) + f(\mathbf{x}_i)\right)^2 \\[3mm]
&+ \frac{\beta_{f_{2}}}{N_f}\sum\limits_{i=1}^{N_f}\left(\nabla\hat{\psi}(\mathbf{x}_i,\tilde{\Theta})\cdot\mathbf{b}^\bot(\mathbf{x}_i) - \hat{\xi}(\mathbf{x}_i,\tilde{\Theta})\right)^2 + \frac{\beta_{f_{3}}}{N_f}\sum\limits_{i=1}^{N_f}\left(\nabla\hat{\psi}(\mathbf{x}_i,\tilde{\Theta})\cdot\mathbf{b}(\mathbf{x}_i) - \exp(\tilde{\varepsilon}^*)\hat{\zeta}(\mathbf{x}_i,\tilde{\Theta})\right)^2 \\[3mm]
&+\frac{\beta_{N_{1}}}{N_n}\sum\limits_{k=1}^{N_n}\left(\exp(\tilde{\varepsilon}^*)\nabla_\bot \hat{\psi}(\mathbf{x}_k,\tilde{\Theta}) \cdot \mathbf{n}(\mathbf{x}_k) +
\nabla_{\|}\hat{\psi}(\mathbf{x}_k,\tilde{\Theta}) \cdot \mathbf{n}(\mathbf{x}_k)\right)^2 \\[3mm]
&+\frac{\beta_{N_{2}}}{N_n}\sum\limits_{k=1}^{N_n}\left(\hat{\xi}(\mathbf{x}_i,\tilde{\Theta})\mathbf{b}^\bot(\mathbf{x}_k) \cdot \mathbf{n}(\mathbf{x}_k)  +
\hat{\zeta}(\mathbf{x}_k,\tilde{\Theta})\mathbf{b}(\mathbf{x}_k) \cdot \mathbf{n}(\mathbf{x}_k)\right)^2,
\end{array}
\end{equation}
\end{subequations}
where the parameters $N_f$, $N_d$, $N_n$ have same meanings as stated in Section~\ref{sec:discrete_LS}. For the identification problem, we merge the Dirichlet boundary term to the observation term. Furthermore, the importance of each term needs to be changed. The weights are redefined as $\beta_{e}$, $\beta_{f_{1}}$, $\beta_{f_{2}}$, $\beta_{f_{3}}$, $\beta_{N_{1}}$, $\beta_{N_{2}}$. It is important to notice that the identification APFOS-LS formulation~\eqref{eq:ident_discrete_LSFormulation_APFOS2D} gives a scalar number $\varepsilon^*$ and a deep neural network at the same time. Moreover, we do not estimate directly the anisotropic strength $\varepsilon$, instead we estimate $\varepsilon^*$ such that $\varepsilon=\exp(\varepsilon^*)$. This change of variable can always make sure the estimation $\varepsilon$ being non-negative. Once the estimation $\varepsilon^*$ is obtained from~\eqref{eq:ident_discrete_LSFormulation_APFOS2D}, we can recover $\varepsilon$ by $\varepsilon=\exp(\varepsilon^*)$.

The non-AP LS scheme for identifying the anisotropic strength $\varepsilon$ is stated in~\ref{sec:nonAPLS_discrete}. We may expect that for $\varepsilon\approx1$, the identification scheme~\eqref{eqs:ident_discrete_LSFormulation_anisotropic} and the AP identification scheme~\eqref{eqs:ident_discrete_LSF_AP2D} provide similar results. While for  $\varepsilon\ll1$, the the identification scheme~\eqref{eqs:ident_discrete_LSFormulation_anisotropic} will no longer work since its corresponding forward problem is ill-posed at the limit $\varepsilon\to0$. In  contrast,  the AP identification scheme~\eqref{eqs:ident_discrete_LSF_AP2D} will still provide accurate estimate of $\varepsilon$ thanks to the well-posedness of APFOS scheme.

\section{Numerical investigation} \label{sec:numresults}

\subsection{Setup for anisotropic elliptic equation}

Two frameworks  are proposed to investigate efficiency of the APFOS scheme.
\paragraph{Setup I} A 2D test case is  manufactured thanks to definition of the $B$-field
\begin{equation}\label{eq:b2d}
\mathbf{B} = \begin{pmatrix}
m\pi\theta(x^2-x)\sin(m\pi z)\\
\pi + \theta (2x - 1) \cos(m \pi z)
\end{pmatrix}
\end{equation}
the associated solution $\phi$ is defined as
\begin{subequations}
\begin{equation}
\phi = \phi_0 + \varepsilon\cos(2\pi z)\sin(\pi x),
\end{equation}
\begin{equation}
\phi_0 = \sin(\omega(\pi x+\theta(x^2-x)\cos(m\pi z))).
\end{equation}
\end{subequations}
The magnetic field $\textbf{b}$ is represented as the normalization of \emph{B}-field $\textbf{B}$, namely $\textbf{b}=\frac{\textbf{B}}{|\textbf{B}|}$.

\paragraph{Setup II} In a 3D case, the $B$-field is defined as
\begin{equation}\label{eq:b3d}
\mathbf{B} = \begin{pmatrix}
2\,\pi\, \left( {x}^{2}-x \right) \sin \left( 3\,\pi\,y \right) \sin
 \left( \pi\,z \right)
\\
\pi\, \left( {x}^{2}-x \right) \sin \left( 3\,\pi\,y \right) \sin
 \left( \pi\,z \right)
\\
2\,\pi+2\, \left( 2\,x-1 \right) \sin \left( 3\,\pi\,y \right) \cos
 \left( \pi\,z \right) +3\,\pi\, \left( {x}^{2}-x \right) \cos \left(
3\,\pi\,y \right) \cos \left( \pi\,z \right)
\end{pmatrix}
\end{equation}
and the solution $\phi$ is given by
\begin{subequations}
\begin{equation}
\phi = \phi_0 + \varepsilon\cos \left( 2\,\pi\,z \right) \sin \left( \pi\,x \right) \sin \left(
\pi\,y \right)
,
\end{equation}
\begin{equation}
\phi_0 =\sin \left( \omega\, \left( \pi\,x+ \left( {x}^{2}-x \right) \sin
 \left( 3\,\pi\,y \right) \cos \left( \pi\,z \right)  \right)
 \right)
.
\end{equation}
\end{subequations}

For both cases, the right-hand side $f$ is analytically computed via
\begin{equation}
f = -\Delta_\bot \phi - \frac{1}{\varepsilon}\Delta_\|\phi.
\end{equation}

To verify the efficiency of the proposed deep neural network-based APFOS scheme (APFOS-DNN), three groups of test cases by different parameter settings are carried out in 2D and 3D frameworks, respectively. In real applications, these model parameters are of physical sense, such as the direction of the magnetic field. In this study, they are set by the following way:

\paragraph{Case 1} $(\theta=0,m=1,\omega=2)$;

\paragraph{Case 2} $(\theta=2,m=1,\omega=2)$;

\paragraph{Case 3} $(\theta=10,m=2,\omega=2)$. \\

In the forward modeling, anisotropic strength $\varepsilon$ is set to $1$, $1e-02$ and $1e-20$ at each \emph{Case} for testing the ability of improving the well-posedness. Fig. \ref{fig:1} illustrates the exact solution $\phi_{e}$ for \emph{Setup I} with $\varepsilon=1e-20$, where $\phi_{e}$ is almost constant along with the direction of the magnetic field $\mathbf{b}$. More precisely, for \emph{Case 1},  the exact solution $\phi_{e}(x,z)\approx \phi_{e}(x)$, while for \emph{Case 2} and \emph{Case 3}, after coordinate transformation, $\phi_{e}(X,Z)\approx \phi_{e}(X)$. Fig. \ref{fig:2} shows the exact solution $\phi_{e}$ for \emph{Setup II} with $\varepsilon=1e-20$.

To describe the numerical performance, the following three types of errors are adopted

\begin{equation}\label{eq:L1}
E_{1} = \frac{\|\phi-\phi_{e}\|_{1}}{\|\phi_{e}\|_{1}}
\end{equation}

\begin{equation}\label{eq:L2}
E_{2} = \frac{\|\phi-\phi_{e}\|_{2}}{\|\phi_{e}\|_{2}}
\end{equation}

\begin{equation}\label{eq:Linf}
E_{\infty} = \frac{\|\phi-\phi_{e}\|_{\infty}}{\|\phi_{e}\|_{\infty}}
\end{equation}
where $\phi_{e}$ denotes the exact solution.

Detailed description on the numerical implementation of the APFOS-DNN method are available in Alg. \ref{alg:1} and \ref{alg:2}, forward modeling and identification problem, respectively.

\begin{algorithm}[htb]\begin{spacing}{1.2}
\caption{APFOS-DNN method for forward modeling}\label{alg:1}
\begin{flushleft}
\hspace*{0.02in}{\bf Architecture:}
Number of discretized grid, hidden layer, neuron, and collocation points of boundary condition and model; activation function; initial weight and bias $\Theta^{0}$; stopping criterion $\mathrm{tol}$ and iteration number $\mathrm{Iter}$.\\
\hspace*{0.02in}{\bf Input of network:}
Location of grid points $\mathbf{x}\in \mathbb{R}^{2}$ or $\mathbf{x}\in\mathbb{R}^{3}$.\\
\hspace*{0.02in}{\bf Output of network:}
Prediction of solution $(\hat{\phi},\hat{\tau},\hat{\sigma})$ or $(\hat{\phi},\hat{\tau},\hat{\chi},\hat{\sigma})$ of 2D/3D.
\end{flushleft}
\begin{algorithmic}[1]
\While{$\|\mathcal{\hat{G}}(\Theta^{k})-\mathcal{\hat{G}}(\Theta^{k-1})\|>\mathrm{tol}$ and $k\leq \mathrm{Iter}$}
\State Compute the gradient of loss function $\mathcal{\hat{G}}$ defined by (\ref{eq:discrete_LSF_AP2D}) or (\ref{eq:discreteAPFOS-LSfunctional}) (2D/3D) with respect to $\Theta^{k}$;
\State Update network parameters $\Theta^{k}$ by the descent optimization algorithm (i.e., L-BFGS);
\State Compute the value of loss function $\mathcal{\hat{G}}(\Theta^{k})$;
    \State $k=k+1$.
\EndWhile
\State Perform the prediction by network $(\hat{\phi},\hat{\tau},\hat{\sigma})$ or $(\hat{\phi},\hat{\tau},\hat{\chi},\hat{\sigma}):=\mathcal{N}(\mathbf{x};\Theta)$ to obtain the numerical solution, where $\Theta$ denotes the optimal parameters by the iteration;
\end{algorithmic}
\end{spacing}
\end{algorithm}

\begin{algorithm}[t]\begin{spacing}{1.2}
\caption{APFOS-DNN method for identification problem}\label{alg:2}
\begin{flushleft}
\hspace*{0.02in} {\bf Architecture:}
Number of discretized grid, hidden layer, neuron, and collocation points of boundary condition, model and available observation; activation function; initial weight and bias $\Theta^{0}$; stopping criterion $\mathrm{tol}$ and iteration number $\mathrm{Iter1,2}$.\\
\hspace*{0.02in} {\bf Input:}
Initial model coefficient $\varepsilon^{\ast,0}$ (first guess).\\
\hspace*{0.02in} {\bf Output:}
Optimal coefficient $\varepsilon^{\ast}_{train}$.\\
\hspace*{0.02in} {\bf Training step: coefficient updation}
\end{flushleft}
\begin{algorithmic}[1]
\While{$\|\mathcal{\hat{G}}(\varepsilon^{\ast,k})-\mathcal{\hat{G}}(\varepsilon^{\ast,k-1})\|>\mathrm{tol}$ and $k\leq \mathrm{Iter1}$}
\State Compute the gradient of loss function $\mathcal{\hat{G}}$ defined by (\ref{eq:ident_discrete_LSF_AP2D}) with respect to $\varepsilon^{\ast,k}$;
\State Update network parameters $\varepsilon^{\ast,k}$ by the descent optimization algorithm (i.e., Adam algorithm);
\State Compute the value of loss function $\mathcal{\hat{G}}(\varepsilon^{\ast,k})$;
    \State $k=k+1$.
\EndWhile
\end{algorithmic}
\begin{flushleft}
\hspace*{0.02in} {\bf Input:}
Location of grid points $\mathbf{x}$ and initial coefficient $\varepsilon^{\ast,0}:=\varepsilon^{\ast}_{train}$.\\
\hspace*{0.02in} {\bf Output:}
Optimal coefficient $\varepsilon$ and numerical solution $(\hat{\phi},\hat{\tau},\hat{\sigma})$.\\
\hspace*{0.02in} {\bf Network optimization step with coefficient updation}
\end{flushleft}
\begin{algorithmic}[1]
\While{$\|\mathcal{\hat{G}}(\Theta^{k};\varepsilon^{\ast,k})-\mathcal{\hat{G}}(\Theta^{k-1};\varepsilon^{\ast,k-1})\|>\mathrm{tol}$ and $k\leq \mathrm{Iter2}$}
\State Compute the gradient of loss function $\mathcal{\hat{G}}$ with respect to $(\Theta^{k};\varepsilon^{\ast,k})$;
\State Update network parameters $(\Theta^{k};\varepsilon^{\ast,k})$ by the descent optimization algorithm (i.e., L-BFGS);
\State Compute the value of loss function $\mathcal{\hat{G}}(\Theta^{k};\varepsilon^{\ast,k})$;
    \State $k=k+1$.
\EndWhile
\State Perform the prediction by network $(\hat{\phi},\hat{\tau},\hat{\sigma}):=\mathcal{N}(\mathbf{x};(\Theta;\varepsilon))$ to obtain the numerical solution, where $\varepsilon:=\exp(\varepsilon^{\ast})$ and $(\Theta;\varepsilon^{\ast})$ denotes the optimal parameters by the iteration.
\end{algorithmic}
\end{spacing}
\end{algorithm}

\begin{figure*}[!t]
\centering
\subfloat[\emph{Case 1}]{\includegraphics[width=2.1in]{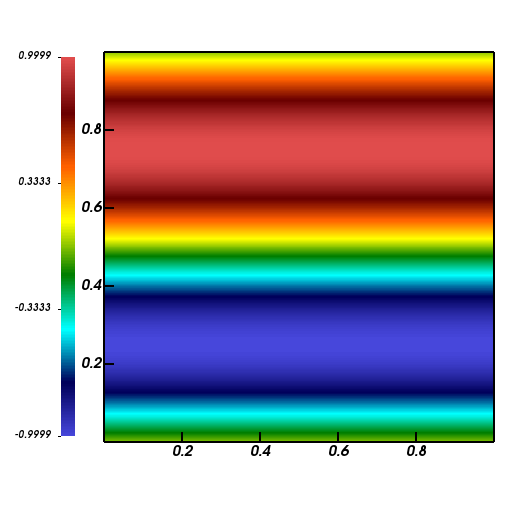}\label{fig:1-1}}
\hfil
\subfloat[\emph{Case 2}]{\includegraphics[width=2.1in]{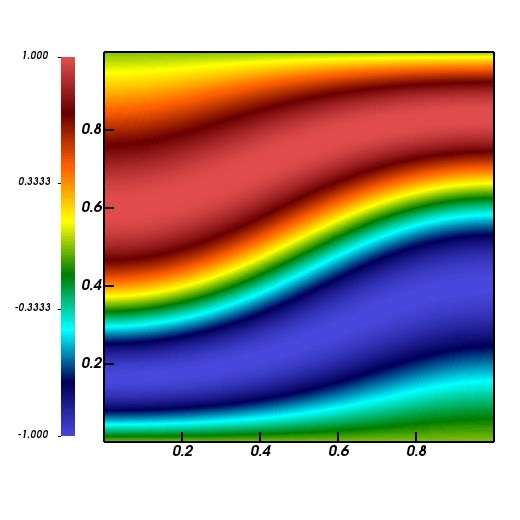}\label{fig:1-2}}
\hfil
\subfloat[\emph{Case 3}]{\includegraphics[width=2.1in]{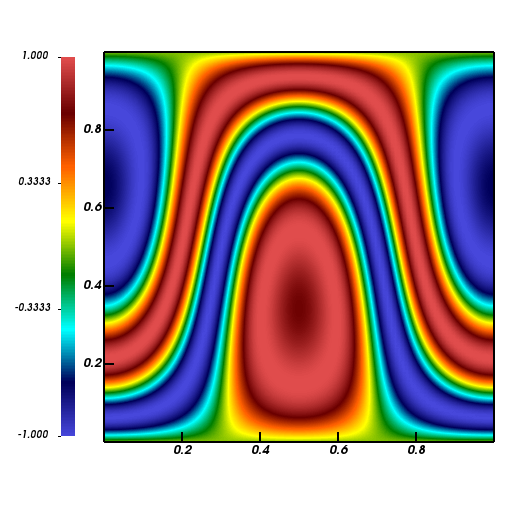}\label{fig:1-3}}
\caption{Exact solution $\phi_{e}$ in three 2D cases with $\varepsilon=1e-20$. From left to right: magnetic field $\mathbf{b}$ aligned with $z$-axis, not aligned with $z$-axis and closed magnetic field.}
\label{fig:1}
\end{figure*}

\begin{figure*}[!t]
\centering
\includegraphics[width=3.0in]{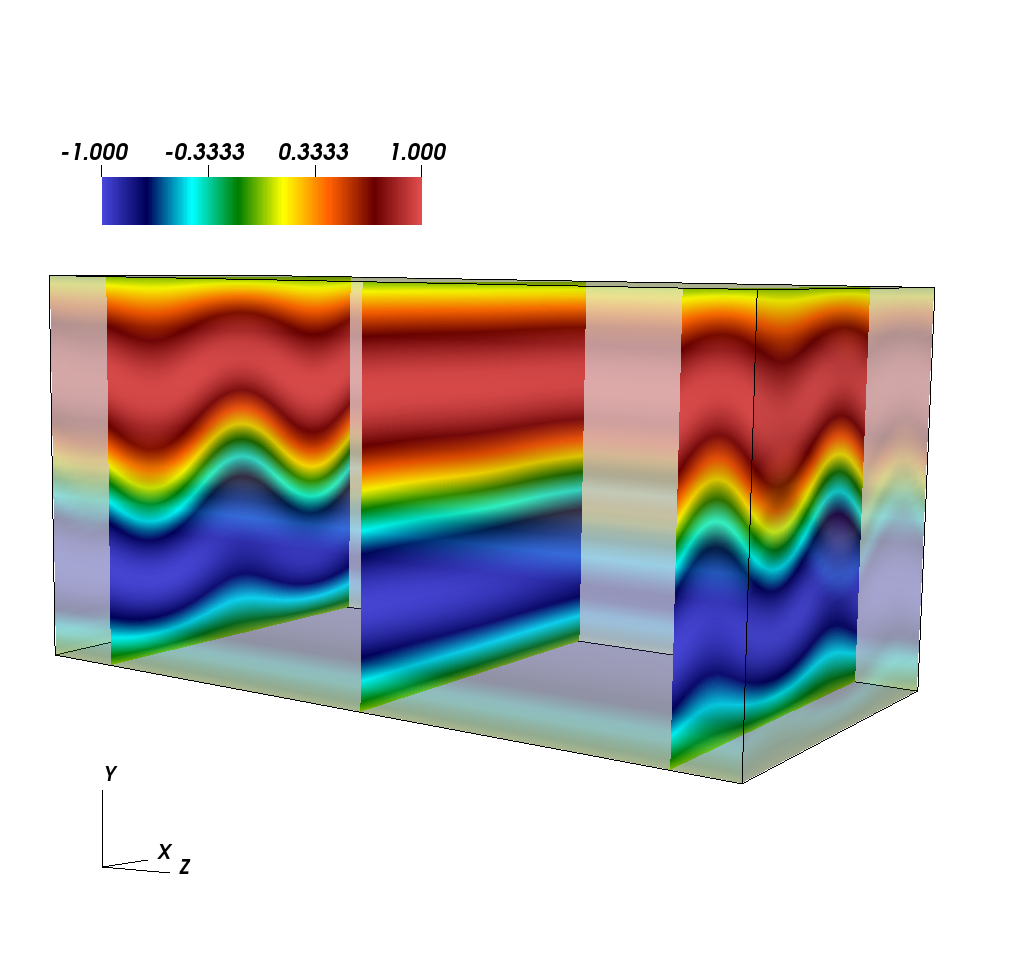}
\caption{Exact solution $\phi_{e}$ in 3D case (\emph{Case 2} with $\varepsilon=1e-20$).}
\label{fig:2}
\end{figure*}

\subsection{Numerical implementation and numerical results}

Herein a numerical scheme based on the deep learning approach is implemented to obtain the numerical solution of anisotropic elliptic equation \eqref{eq:anisotropic_pb}. Our goal aims at finding the optimal solution of the least-squares problem described by \eqref{eqs:discrete_LSF_AP2D}, \eqref{eqs:discreteAPFOS-LSfunctional} and \eqref{eqs:ident_discrete_LSF_AP2D}. Therefore, the first-order derivatives with respect to the network parameters $\Theta$ (i.e., weight $\mathcal{W}$ and bias $\mathcal{B}$), need to be computed. In addition, we observe that the PDE-based network involves differential operators from the forward model, such as the first-order derivatives with respect to the coordinates, $\partial _{x}(\cdot)$ and $\partial _{y}(\cdot)$. Herein, the dominant automatic differentiation (AD) technique is employed that allows efficient computation of the derivatives. The so-called back-propagation algorithm belonging to AD has been popular among large scaled inverse problems, especially for deep learning \cite{ADindeep}. L-BFGS \cite{1989On} and stochastic gradient descent (SGD) algorithm (i.e., Adam algorithm) \cite{2014Adam} for large data-sets are used for updating the network parameters.

In the next, we focus on the settings of the architecture of the deep neural network in 2D/3D forward modeling and identification test cases and numerical results.

\subsubsection{Data-driven solutions in 2D}

It is well known that the selection of the hyper-parameters (i.e., the number of layer and neurons) will be crucial to the DNN-based scheme. In general, it could have an influence on the accuracy of the numerical solution and the convergence of the algorithm. In this study, the hyper-parameters could be derived from a twin experiment \cite{2010Assimilation}, in which the results by the setting of $3$, $4$, $5$ hidden layers and $40$, $60$, $80$ neurons are compared. It can be concluded from Table \ref{tab:testnumber}, when the network contains 4/4 layers and 40/60 neurons, the best results can be obtained by non-AP-DNN and APFOS-DNN methods, respectively. Therefore, in 2D test cases, such robust hyper-parameters are adopted.

The solution is discretized at a square domain $\Omega=[0,1]\times[0,1]$ with $N\times N=100\times100$ grids. Both the numbers of the collocation points at the Dirichlet boundary $\Gamma_D=\{0,1\}\times[0,1]$ and Neumann boundary  $\Gamma_N=[0,1]\times\{0,1\}$ are set to $N_{d}=N_{n}=100$, as well as $N_{f}=4000$ inside the domain. According to the definition in \cite{Raissi2019}, collocation points are selected randomly and only the information on such points are available. We take the weights $(\beta_D, \beta_N)=(N_d, 1)$ and $(\alpha_D, \alpha_N)=(N_d, 1)$ in \eqref{eq:discrete_LSF_AP2D} and \eqref{eq:discrete_LSF_anisotropic} respectively.

Table \ref{tab:2dcase} shows the errors between the numerical solution obtained by the DNN-based method and the exact one. Meanwhile, the results at $1000$ and $10000$ iteration numbers are also recorded for analyzing the convergence (Normal and bold fonts). In test \emph{Case 1}, it can be seen that the $E_{2}$ errors easily reach at the magnitude of $1e-04$ and $1e-03$ by using the non-AP-DNN approach when $\varepsilon=1$ and $1e-02$, respectively. The APFOS-DNN method outperforms the non-AP-based one for different $\varepsilon$, specially for $\varepsilon\ll 1$ and the $E_{2}$ error could be less than $1e-04$. More clearly, the numerical solutions at $x=0.5$ are displayed by Fig. \ref{fig:3-1}-\ref{fig:3-3}. Those numerical results demonstrate that the accuracy of the APFOS scheme is not relevant to anisotropic strength $\varepsilon$. For the non-AP scheme, the numerical error will be involved when $\varepsilon \ll1$ and it shows the non-AP scheme diverges for $\varepsilon=1e-20$.

In test \emph{Case 2}, it could be found from Table \ref{tab:2dcase} that the solution with desired accuracy could not be reached by the non-AP-DNN method in the case of $\varepsilon=1e-02$, which can be explained as numerical pollution as mentioned in~\cite{yang2020Preserving}. The algorithm converges at around 5000 iterations and the value of loss function is $8.8e-05$. In the contrary, there is no doubt that the proposed APFOS-DNN approach is as effective as ever. The $E_{2}$ error could reach at the magnitude of $1e-04$. The difference could be found in Fig \ref{fig:3-4} and \ref{fig:3-5}. Non-AP method cannot approximate the peaks and troughs of the solution well. Fig \ref{fig:3-5} and \ref{fig:3-6} demonstrate that combining with the AP scheme, neural network is capable of maintaining the accuracy uniformly with respect to $\varepsilon$.

Additionally, we focus on a complex example (\emph{Case 3}) that contains the closed magnetic field. The difficulty of such problem lies in the singularity phenomenon appears in the center of closed magnetic field. There are lots of singularity points at which the gradient of the field $\mathbf{b}$ tends to infinity. Furthermore, the exact solution covers steep gradient characteristics around closed magnetic field (as shown in Fig. \ref{fig:1-3}). Therefore, it is necessary to make appropriate adjustments on the architecture of the neural network. Particularly, in the case of $\varepsilon=1$, herein a neural network covering four hidden layers with $80$ neurons is used in APFOS-DNN method. Meanwhile, the numbers of the collocation points inside the domain are changed to $N_{f}=5000$.

It can be seen from Table \ref{tab:2dcase} that for $\varepsilon=1$, non-AP-DNN method performs well. The $E_{2}$ error reaches at the magnitude of $1e-03$. However, in the case of $\varepsilon=1e-02$, the $E_{2}$ error stops at the magnitude of $1e-01$. Fig. \ref{fig:4-3} suggests the numerical solution could not approximate the exact one well near the boundary. It can be concluded from Table \ref{tab:2dcase} and Fig. \ref{fig:4-4}-\ref{fig:4-5} that the APFOS-DNN method could achieve the numerical approximation with expected accuracy, whether inside the domain of closed magnetic field or at the location where the solution contains steep gradient characteristics. Compared with previous examples, the $E_{\infty}$ errors of \emph{Case 3} (as shown in Table \ref{tab:2dcase}) imply that approximating the solution at the singularity point will inevitably cause larger errors.

In conclusion, compared with the non-AP-DNN method, the APFOS-DNN method is capable of overcoming numerical pollution of the anisotropic elliptic equation when $\varepsilon$ is small, thus it is uniformly accurate with respect to $\varepsilon$. In the next section, we mainly care about its ability of dealing with 3D problem.

\begin{table}[]
\renewcommand{\arraystretch}{1.5}
    \caption{Comparison of $E_{2}$ errors by using different numbers of neurons and hidden layers. Normal and bold fonts represent the $E_{2}$ errors of non-AP-DNN method at \emph{Case 1} with $\varepsilon=1e-02$ and APFOS-DNN method at \emph{Case 2} with $\varepsilon=1e-20$.}
    \vspace{10pt}
    \centering
    \begin{tabular}{c|cccc}
        \hline
        \diagbox{Neuron}{Error}{Layer} & 3 & 4 & 5 \\
        \hline
        \multirow{2}{1em}{40}  & 8.97e-03 & \underline{5.25e-03}  &  7.18e-03          \\
          & \textbf{1.36e-03}  & \textbf{1.69e-03}               & \textbf{1.01e-03}                        \\ \hline
        \multirow{2}{1em}{60}  & 8.27e-03     & 1.62e-02             & 1.17e-02                       \\
          & \textbf{1.47e-03}     & \underline{\textbf{7.19e-04}}              & \textbf{1.02e-03}                           \\ \hline
           \multirow{2}{1em}{80}  & 3.77e-02     & 5.41e-02            & 2.58e-02                         \\
          & \textbf{8.03e-04}    & \textbf{9.43e-04}          & \textbf{8.21e-04}                          \\
        \hline
    \end{tabular}
    \label{tab:testnumber}
\end{table}

\begin{table}[]
\renewcommand{\arraystretch}{1.5}
    \caption{Comparison of errors in 2D case. Normal and bold fonts represent the errors at $1000$ and $10000$ iterations. From top to bottom for each \emph{Case}: $E_{1}$, $E_{2}$ and $E_{\infty}$.}
    \vspace{10pt}
    \centering
    \addtolength{\tabcolsep}{-2pt}
    \begin{tabular}{c|c|cc|cc|ccc}
        \hline
        \diagbox{Scheme}{Error}{Epsilon}& Case & \multicolumn{2}{c}{1} & \multicolumn{2}{c}{1e-02} & \multicolumn{2}{c}{1e-20}  \\
        \hline
       \multirow{9}*{\minitab[c]{Non-AP-DNN}} & \multirow{3}*{\minitab[c]{\emph{Case 1}}} & 2.53e-03 &\textbf{3.44e-04}   & 8.06e-02 &\textbf{5.07e-03} & -  & -           \\
       &&2.96e-03 &\textbf{4.30e-04}  & 7.83e-02 &\textbf{5.25e-03}             & -    & -                      \\
       & & 1.01e-02 &\textbf{1.29e-03}   & 7.83e-02 &\textbf{6.29e-03}             & -     & -                     \\ \cline{2-8}
      &\multirow{3}*{\minitab[c]{\emph{Case 2}}}  & 2.59e-03 &\textbf{4.04e-04}   & 8.53e-01 &\textbf{1.68e-02}             & -   & -                        \\
       & & 2.61e-03 &\textbf{4.90e-04}   & 9.07e-01& \textbf{1.67e-02}             & -   & -                        \\
       & & 8.92e-03 &\textbf{1.42e-03}   & 1.16 & \textbf{1.91e-02}             & -    & -                     \\ \cline{2-8}
     &\multirow{3}*{\minitab[c]{\emph{Case 3}}}  & 2.10e-01 &\textbf{2.66e-03}   & 9.95e-01 &\textbf{9.55e-02}             & -    & -                       \\
      &  & 1.98e-01 &\textbf{3.60e-03}   & 9.98e-01 &\textbf{2.34e-01}             & -    & -                       \\
      & & 2.80e-01 &\textbf{1.37e-02}   & 1.45 &\textbf{9.82e-01}             & -     & -                      \\ \hline

    \multirow{9}*{\minitab[c]{APFOS-DNN}} &\multirow{3}*{\minitab[c]{\emph{Case 1}}} & 1.14e-03 &\textbf{9.80e-05}   & 1.96e-03 &\textbf{3.16e-04}            & 2.20e-03 &\textbf{4.74e-04}                        \\
      &   & 1.17e-03 &\textbf{1.06e-04}   & 2.27e-03 &\textbf{3.55e-04}             & 2.50e-03 &\textbf{5.56e-04}                          \\
       &   & 2.46e-03 &\textbf{2.02e-04}   & 5.93e-03 &\textbf{9.90e-04}            & 7.49e-03 &\textbf{1.59e-03}                           \\ \cline{2-8}
     &\multirow{3}*{\minitab[c]{\emph{Case 2}}} & 2.06e-03 &\textbf{1.52e-04}   & 3.54e-03 &\textbf{6.71e-04}            & 7.43e-03 &\textbf{6.30e-04}               \\
      &   & 2.08e-03 &\textbf{1.60e-04}    & 3.95e-03& \textbf{7.65e-04}             & 8.96e-03 &\textbf{7.19e-04}                          \\
      &    & 3.03e-03 &\textbf{3.09e-04}   & 8.48e-03& \textbf{1.80e-03}            & 1.91e-02 &\textbf{1.95e-03}                          \\ \cline{2-8}
      &\multirow{3}*{\minitab[c]{\emph{Case 3}}} & 1.32e-01 &\textbf{8.01e-03}   & 7.73e-02 &\textbf{2.95e-02}            & 1.23e-01 &\textbf{2.06e-02}                        \\
       &  & 1.51e-01 &\textbf{8.07e-03}   & 8.83e-02 &\textbf{3.67e-02}             & 1.51e-01 &\textbf{2.57e-02}                          \\
       &   & 3.25e-01 &\textbf{1.84e-02}   & 2.29e-01 &\textbf{1.09e-01}            & 4.35e-01 &\textbf{8.20e-02}                         \\
        \hline
    \end{tabular}
    \label{tab:2dcase}
\end{table}

\begin{figure*}[!t]
\centering
\subfloat[Non-AP-DNN, \emph{Case 1} with $\varepsilon=1e-02$]{\includegraphics[width=1.8in]{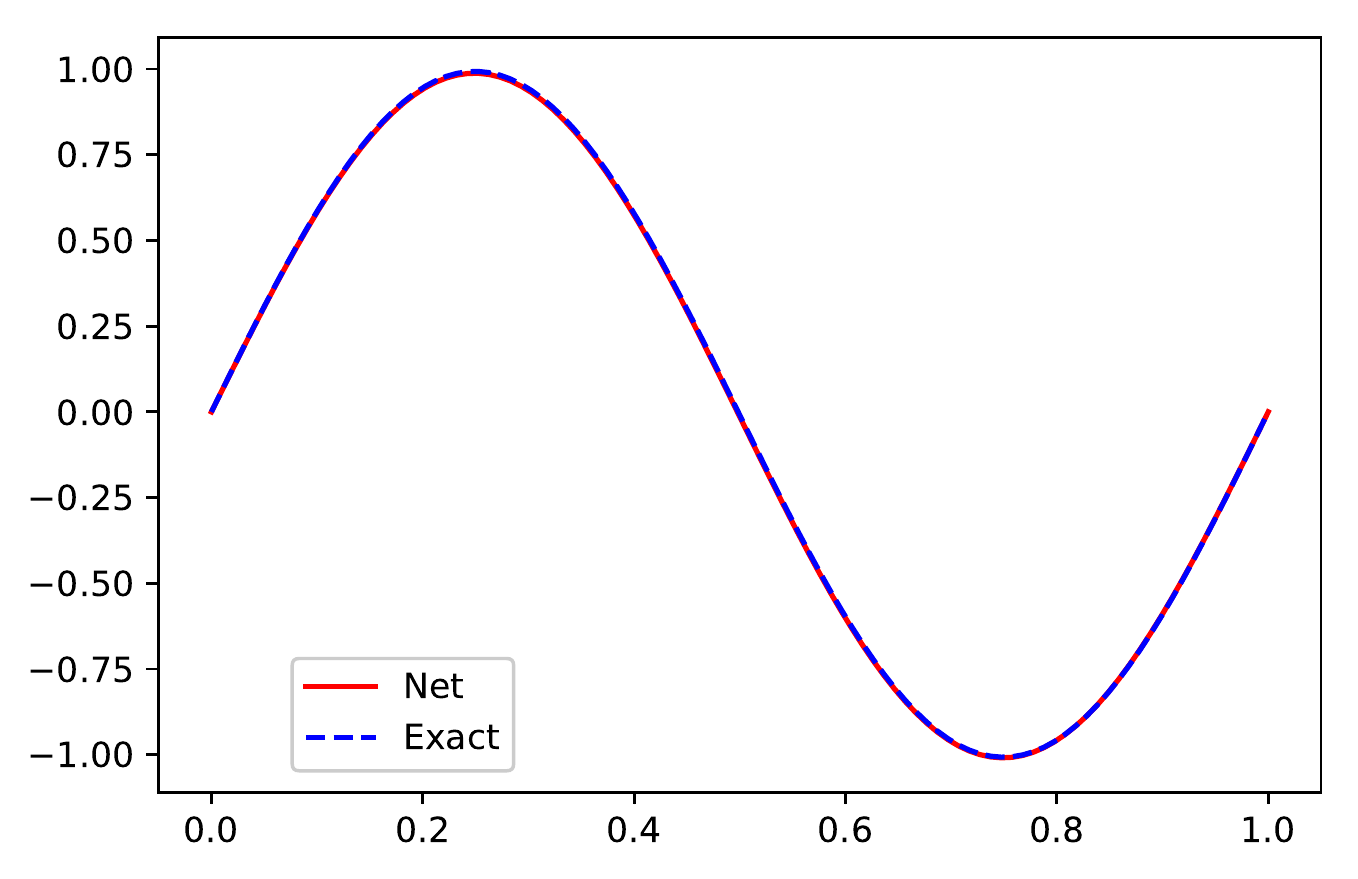}\label{fig:3-1}}
\hfil
\subfloat[APFOS-DNN, \emph{Case 1} with $\varepsilon=1e-02$]{\includegraphics[width=1.8in]{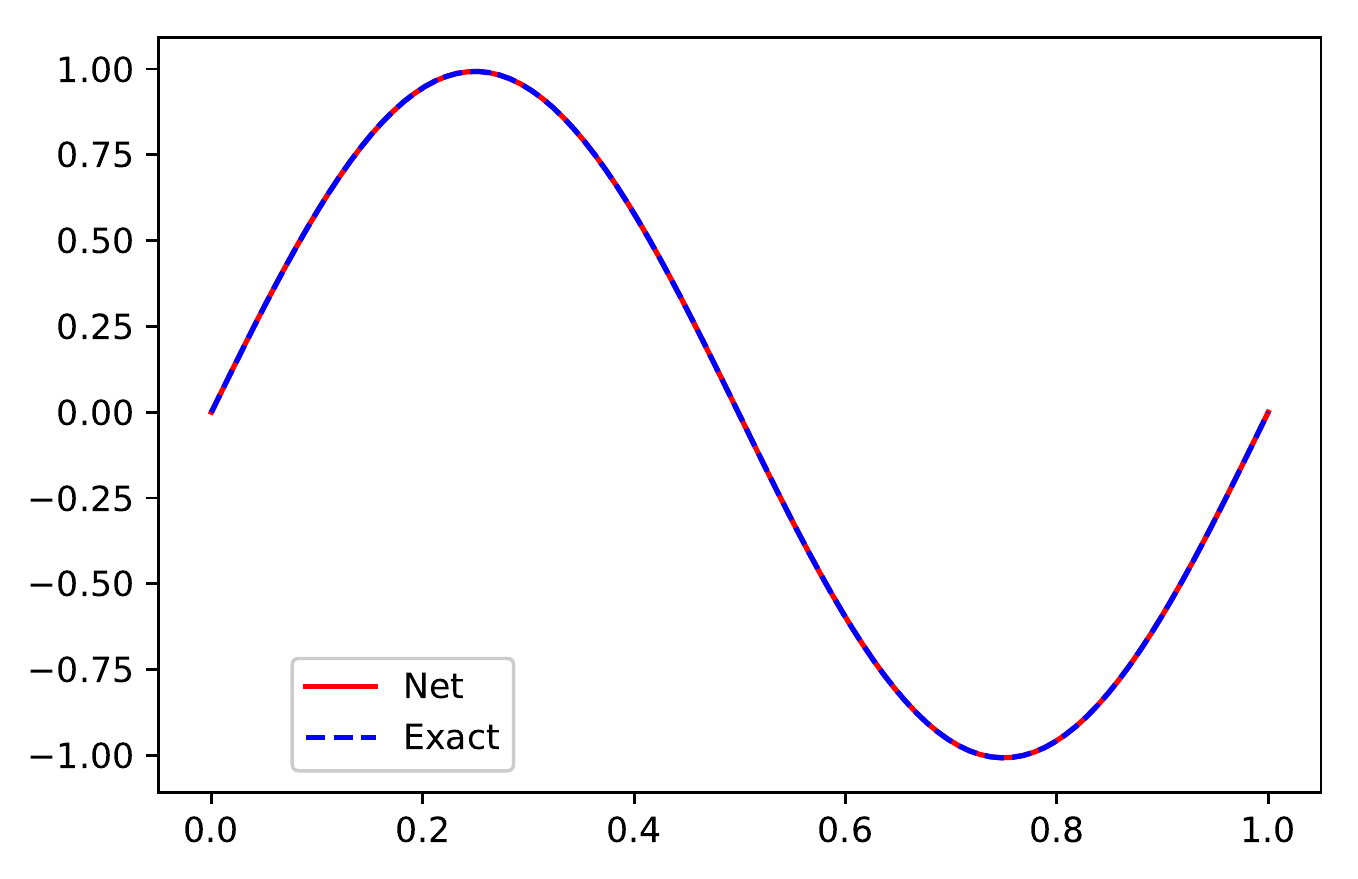}\label{fig:3-2}}
\hfil
\subfloat[APFOS-DNN, \emph{Case 1} with $\varepsilon=1e-20$]{\includegraphics[width=1.8in]{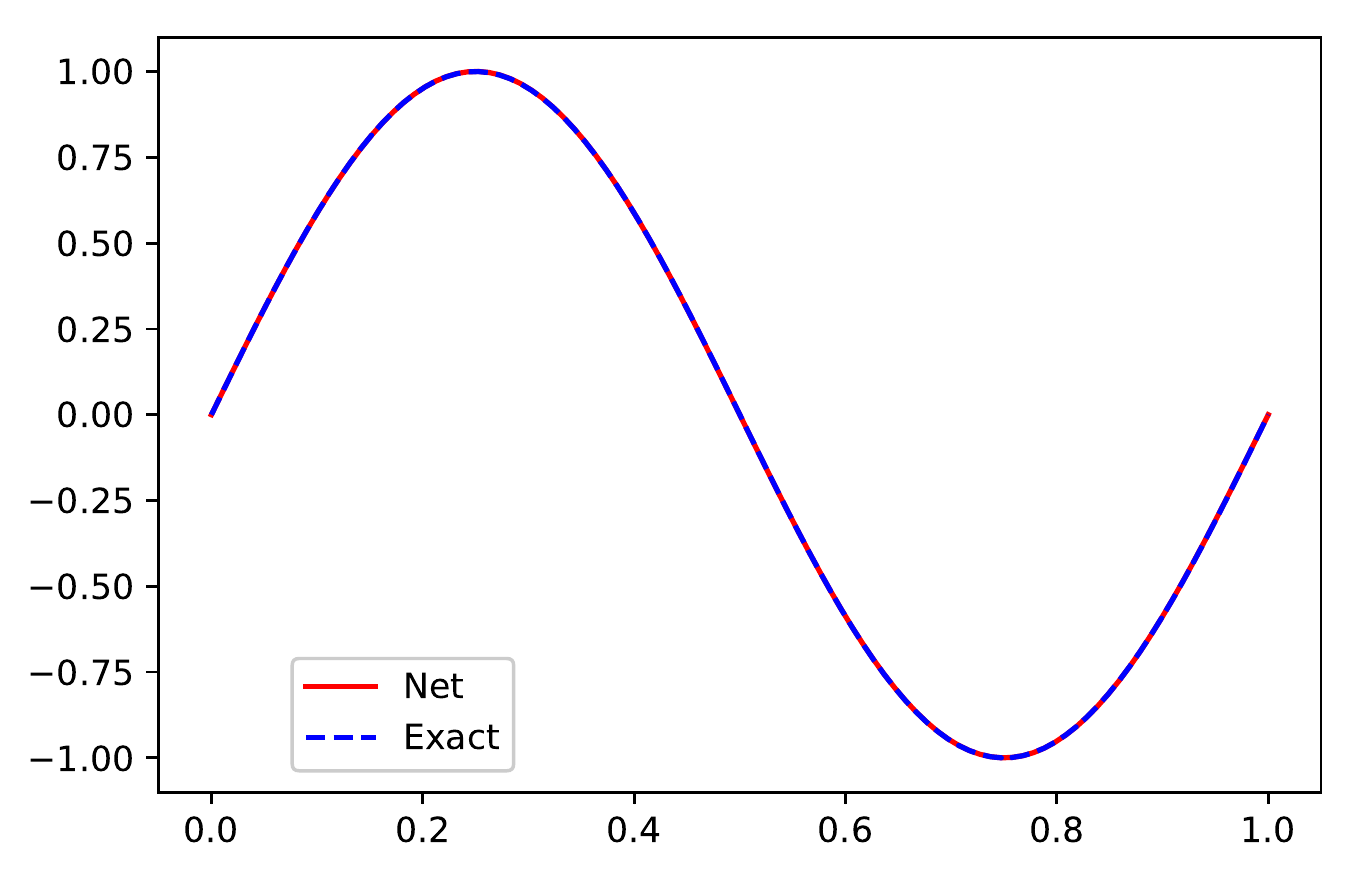}\label{fig:3-3}}
\vfil
\subfloat[Non-AP-DNN, \emph{Case 2} with $\varepsilon=1e-02$]{\includegraphics[width=1.8in]{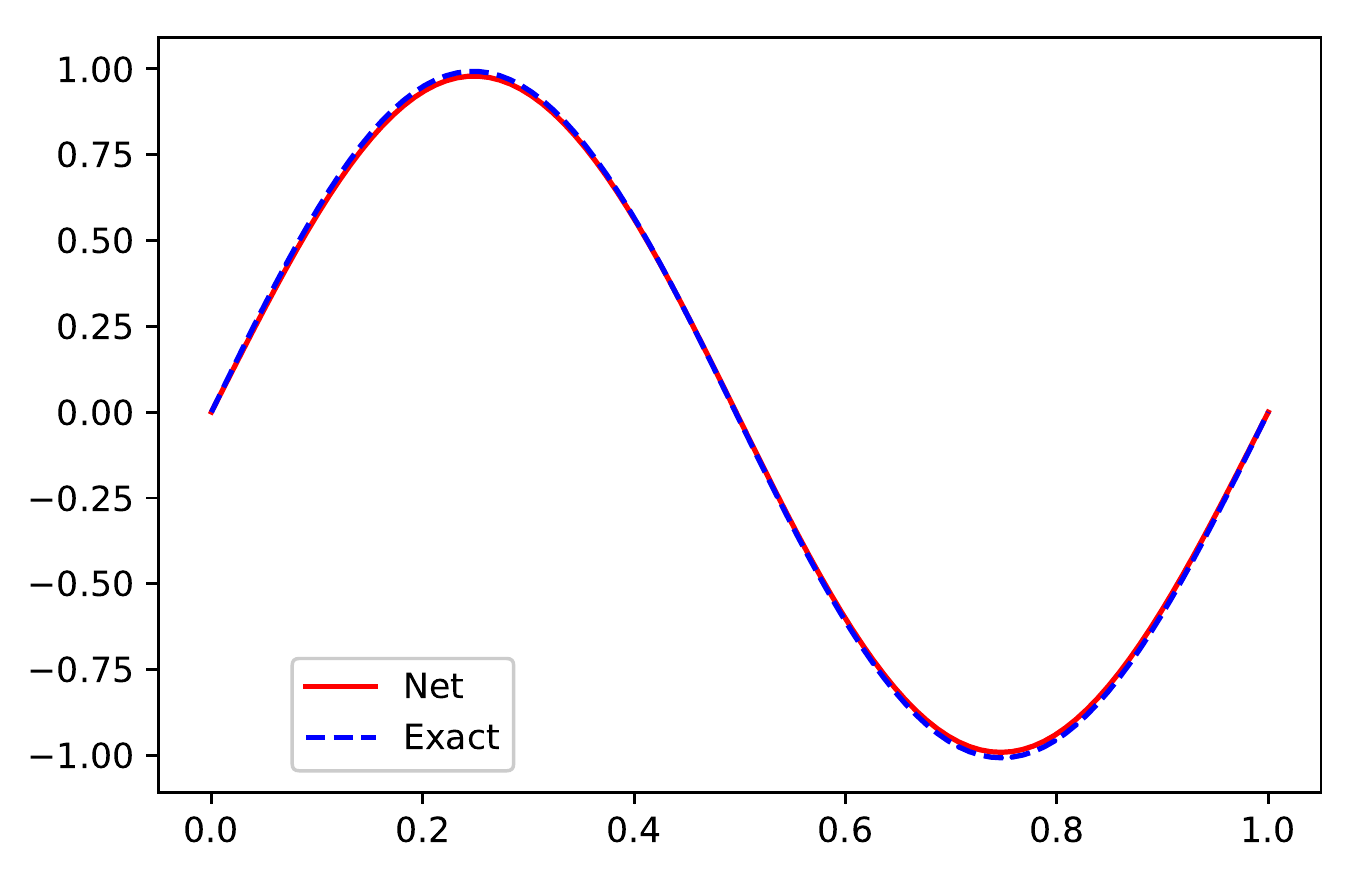}\label{fig:3-4}}
\hfil
\subfloat[APFOS-DNN, \emph{Case 2} with $\varepsilon=1e-02$]{\includegraphics[width=1.8in]{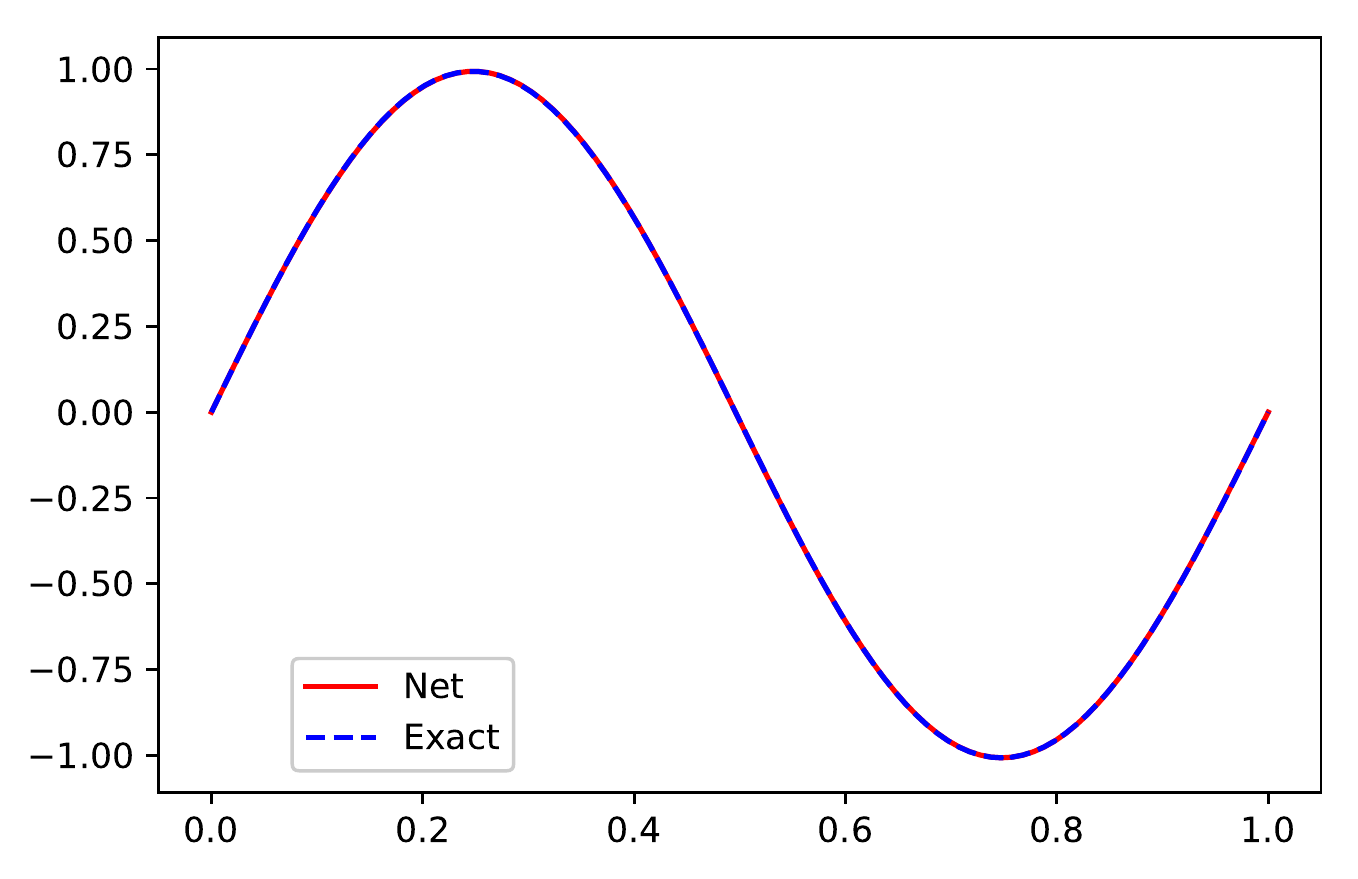}\label{fig:3-5}}
\hfil
\subfloat[APFOS-DNN, \emph{Case 2} with $\varepsilon=1e-20$]{\includegraphics[width=1.8in]{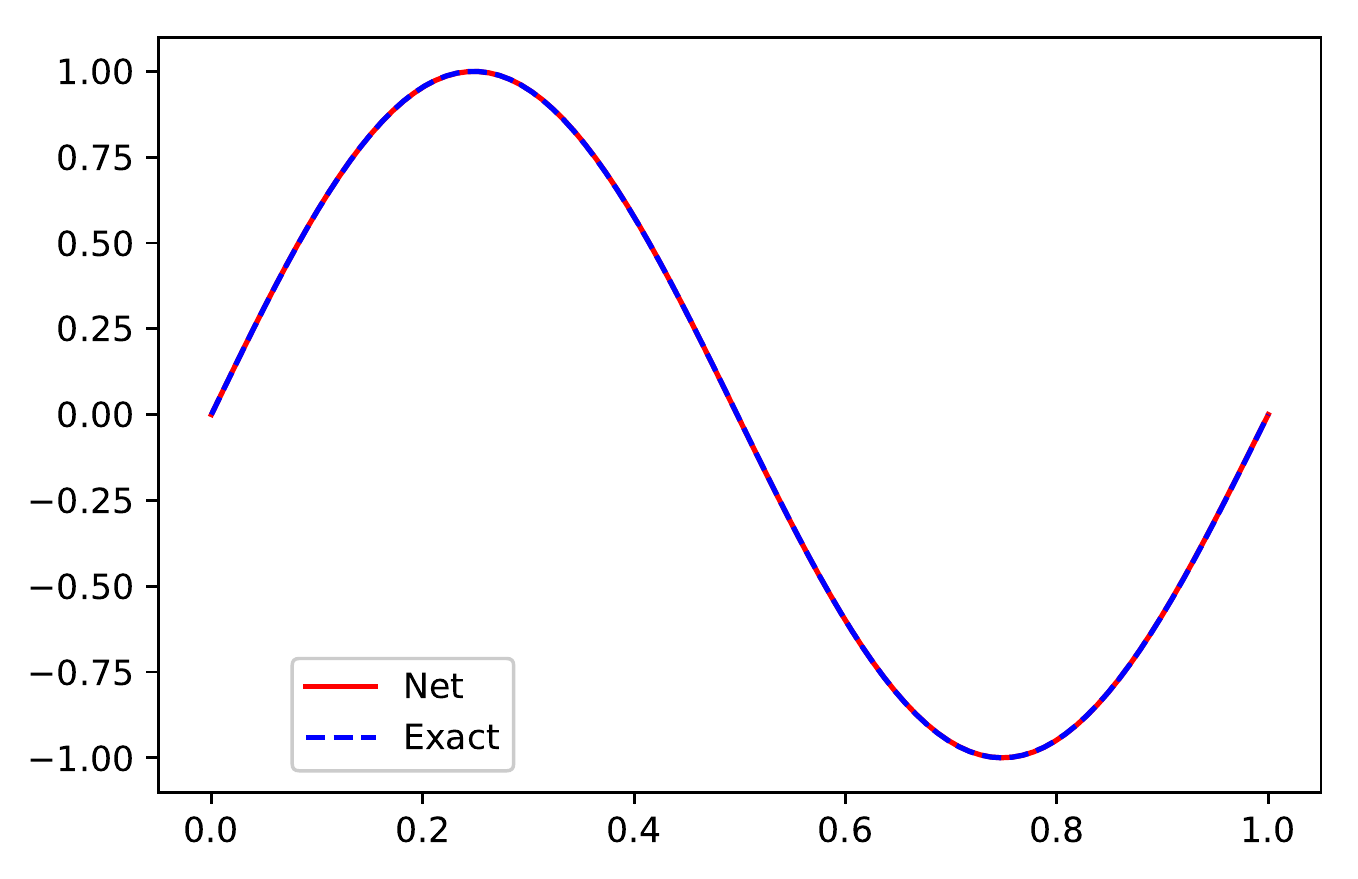}\label{fig:3-6}}
\caption{Comparison of numerical solution in 2D case obtained by non-AP-DNN and APFOS-DNN methods (\emph{Case 1} and \emph{2}, $x=0.5$).}
\label{fig:3}
\end{figure*}

\begin{figure*}[!t]
\centering
\subfloat[Non-AP-DNN, \emph{Case 3} with $\varepsilon=1$]{\includegraphics[width=1.8in]{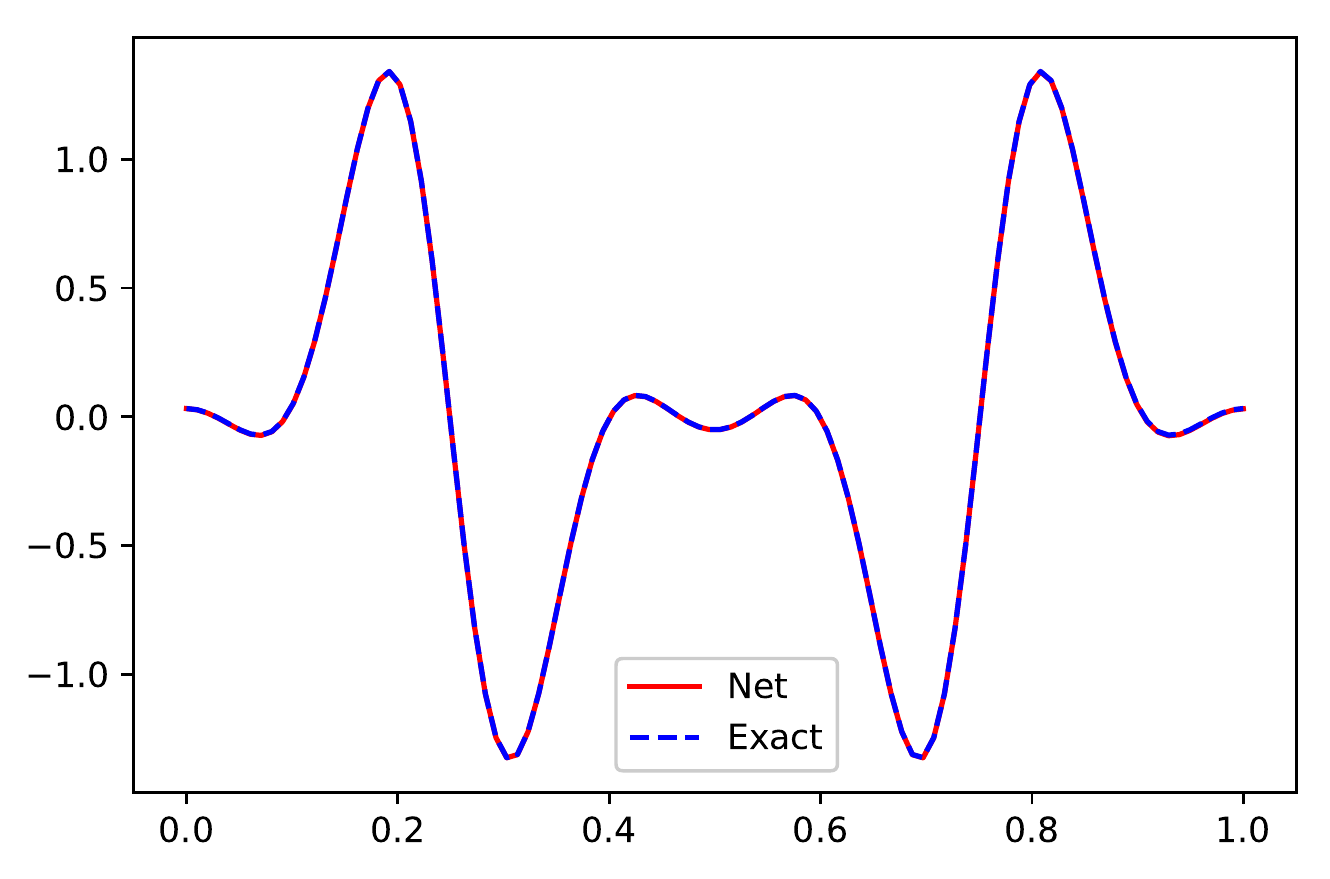}\label{fig:4-1}}
\hfil
\subfloat[APFOS-DNN, \emph{Case 3} with $\varepsilon=1$]{\includegraphics[width=1.8in]{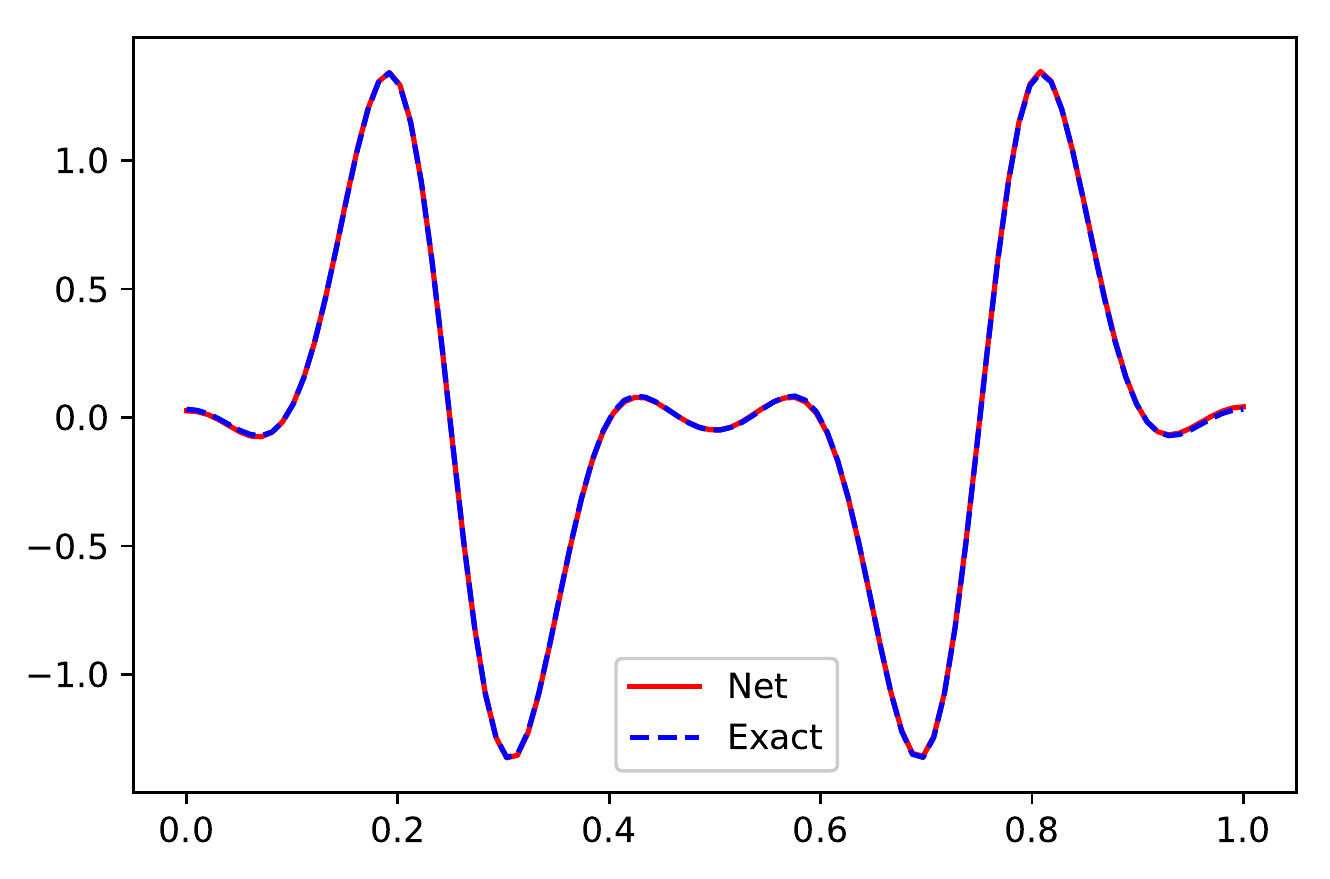}\label{fig:4-2}}
\vfil
\subfloat[Non-AP-DNN, \emph{Case 3} with $\varepsilon=1e-02$]{\includegraphics[width=1.8in]{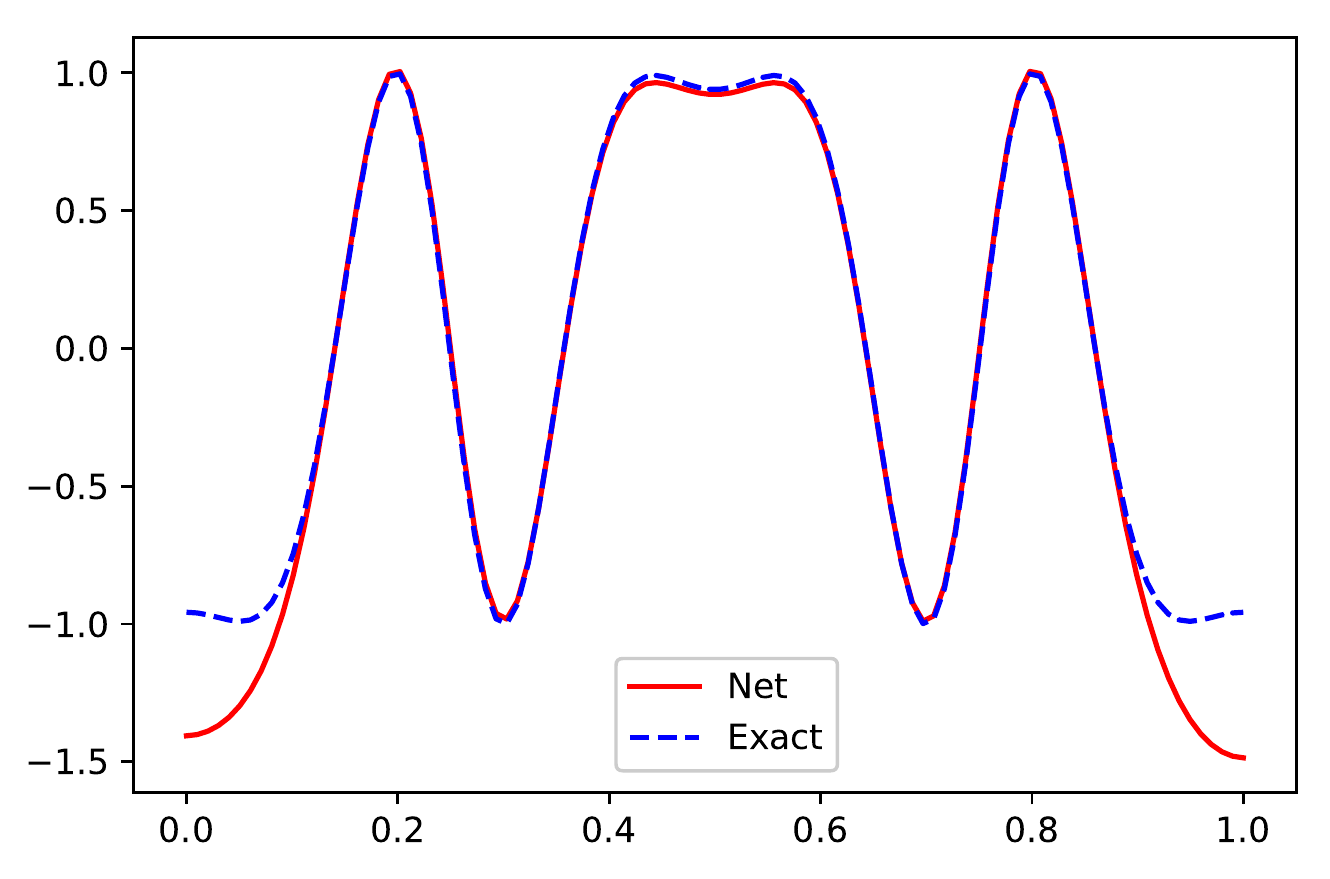}\label{fig:4-3}}
\hfil
\subfloat[APFOS-DNN, \emph{Case 3} with $\varepsilon=1e-02$]{\includegraphics[width=1.8in]{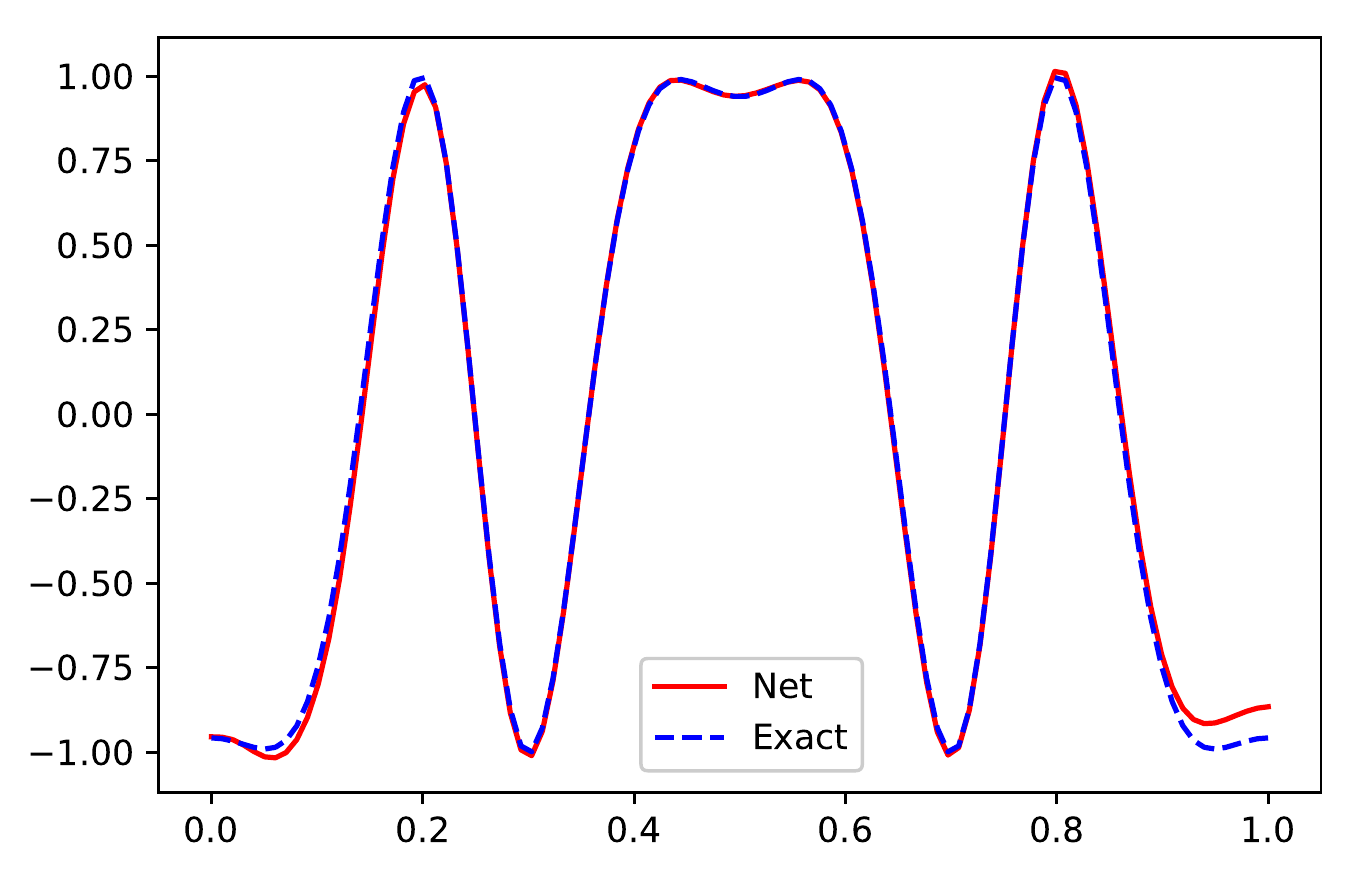}\label{fig:4-4}}
\vfil
\subfloat[APFOS-DNN, \emph{Case 3} with $\varepsilon=1e-20$]{\includegraphics[width=1.8in]{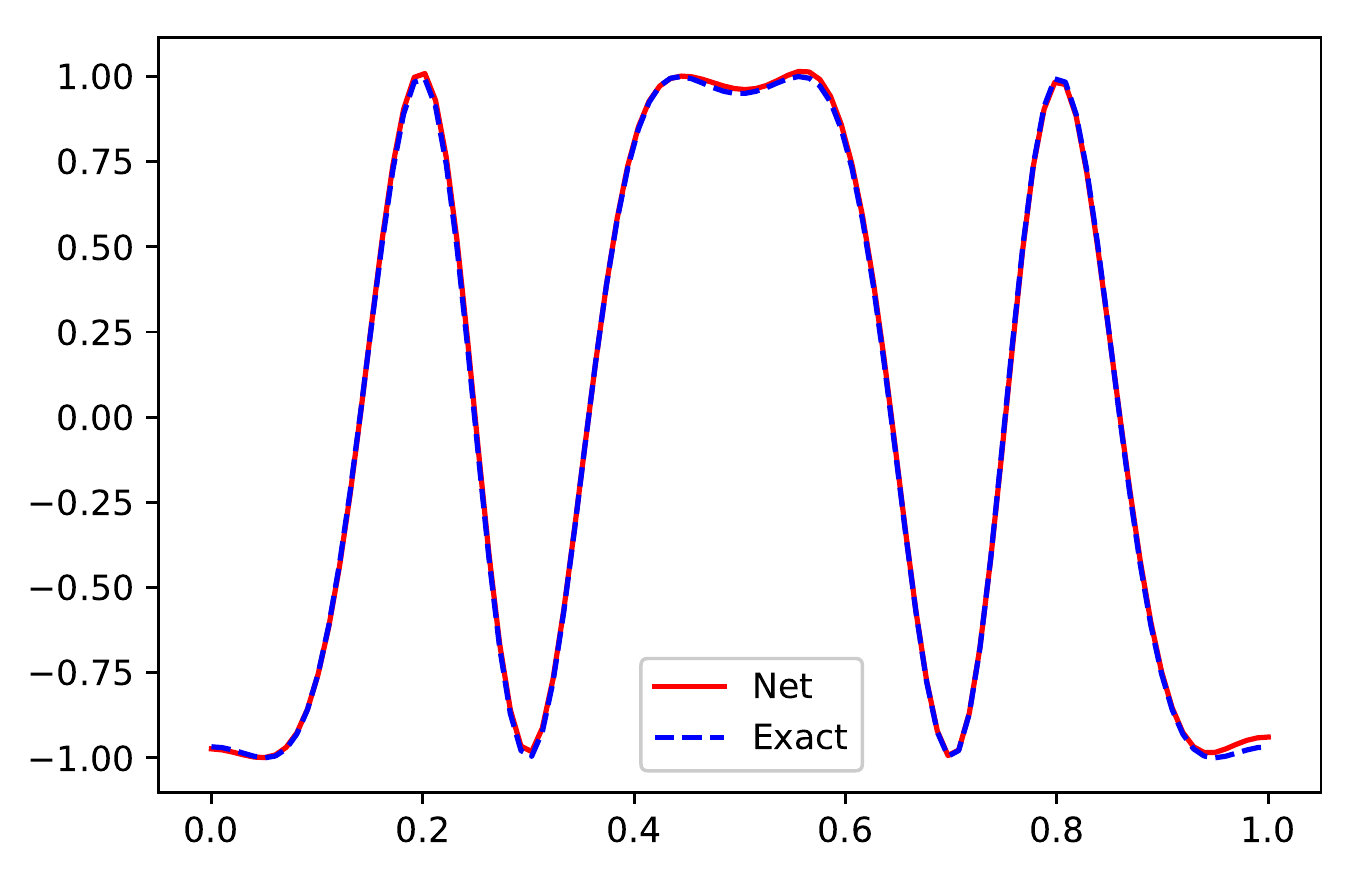}\label{fig:4-5}}
\caption{Comparison of numerical solution in 2D case obtained by non-AP-DNN and APFOS-DNN methods (\emph{Case 3}, $x=0.5$).}
\label{fig:4}
\end{figure*}

\subsubsection{Data-driven solutions in 3D}

In real applications, the 3D anisotropic elliptic equation could better describe the physical phenomenon. Therefore, herein we mainly care about the numerical implementation and performance of the proposed APFOS-DNN approach in 3D cases. The 3D numerical example established by \emph{Setup II} with parameter setting of \emph{Case 2} is investigated.

The 3D anisotropic elliptic equation is solved numerically at a cubic domain $\Omega=[0,1]\times[0,1]\times[0,1]$ with $N\times N\times N=50\times50\times50$ grids. The neural network covering four hidden layers with $60$ neurons at each layer is used in the non-AP-DNN and APFOS-DNN methods. The numbers of the collocation points at the Dirichlet boundary $\Gamma_D=\{\{0,1\}\times[0,1]\times[0,1],\,[0,1]\times\{0,1\}\times[0,1]\}$ and Neumann boundary $\Gamma_N=[0,1]\times[0,1]\times\{0,1\}$ are all set to $N_{d}=N_{n}=1000$, as well as $N_{f}=10000$ inside the domain. In the 3D case, we take the weights $(\beta_D, \beta_N)=(N_d, 1)$ in \eqref{eq:discreteAPFOS-LSfunctional}.

As expected, some impressive numerical results could be obtained by the neural network-based method in 3D case. Table \ref{tab:3dcase} records three types of errors at $5000$ and $10000$ iterations (Normal and bold fonts) for different $\varepsilon$ and methods. It can be observed that both approaches could produce the solution with $E_{2}$ errors of the magnitude of $1e-03$, except the non-AP-DNN method in the case of $\varepsilon=1e-02$ and $1e-20$. Because of the well-posedness of the APFOS scheme, the value of loss function $\mathcal{\hat{G}}$ defined by \eqref{eq:discreteAPFOS-LSfunctional} could drop to $0.007$ after $10000$ iterations when $\varepsilon=1e-20$. In addition, an obvious advantage over the classical methods is that the APFOS-DNN method could work well when the numbers of collocation points are far less than the discretized grids used for the classical methods. Notice that the numbers of collocation points  are $4000$  in 2D cases (40\% of uniform grid), and  they are only $10000$ in 3D cases (8\% of uniform grid).

Fig. \ref{fig:5} displays the corresponding numerical results at the location $(x=0.5,z=0)$ (Left column), $(x=0.5,z=0.5)$ (Middle column), and $(x=0.5,z=1)$ (Right column), respectively. The unsatisfied results shown in Fig. \ref{fig:5-1}-\ref{fig:5-3} indicate the ineffectiveness of the non-AP-DNN method in the case of small $\varepsilon$. The well-posed problem by using the APFOS scheme yields promising results (as shown in Fig. \ref{fig:5-4}-\ref{fig:5-9}). As it is said in \cite{Raissi2019}, the DNN-based method could achieve good prediction accuracy given a sufficiently expressive neural network architecture and a sufficient number of collocation points if the partial differential equation is well-posed. However, it can be found from Table \ref{tab:3dcase} that the $E_{2}$ errors are about 10 times larger than those in 2D case. A reasonable explanation is that the random selection of the collocation points leads to lose the location where the exact solution contains complex structures. In the future work, we will pay more attention to the adaptive grids method that takes the structure of the solution into consideration.

As a result, owing to the nice performance on 3D problem, the neural network-based APFOS method will be a good choice of strategy in real applications.

\begin{table}[]
\renewcommand{\arraystretch}{1.5}
    \caption{Comparison of errors in 3D case. Normal and bold fonts represent the errors at $5000$ and $10000$ iterations. From top to bottom for each method: $E_{1}$, $E_{2}$ and $E_{\infty}$.}
    \vspace{10pt}
    \centering
    \begin{tabular}{c|cc|cc|ccc}
        \hline
        \diagbox{Scheme}{Error}{Epsilon} & \multicolumn{2}{c}{1} & \multicolumn{2}{c}{1e-02} & \multicolumn{2}{c}{1e-20}  \\
        \hline
        \multirow{3}*{\minitab[c]{Non-AP-DNN}}  &  2.78e-03 & \textbf{1.75e-03}   & 9.98e-01 & \textbf{1.01}             & - & -                          \\
        & 3.28e-03 &\textbf{1.99e-03}   & 1.14 & \textbf{1.14}             & -    & -                      \\
          & 1.26e-02 &\textbf{6.97e-03}   & 1.79 & \textbf{1.75}             & -   & -                       \\ \hline
       \multirow{3}*{\minitab[c]{APFOS-DNN}}  & 4.07e-03 &\textbf{2.08e-03}   & 9.24e-03 &\textbf{5.95e-03}            & 9.81e-03 &\textbf{6.57e-03}                           \\
         & 4.46e-03 &\textbf{2.28e-03}   & 1.20e-02 &\textbf{7.30e-03}             & 1.23e-02 &\textbf{8.32e-03}                          \\
          & 8.73e-03 &\textbf{4.71e-03}   & 3.79e-02 &\textbf{2.01e-02}            & 4.86e-02 &\textbf{2.65e-02}                           \\
        \hline
    \end{tabular}
    \label{tab:3dcase}
\end{table}

\begin{figure*}[!t]
\centering
\subfloat[]{\includegraphics[width=1.8in]{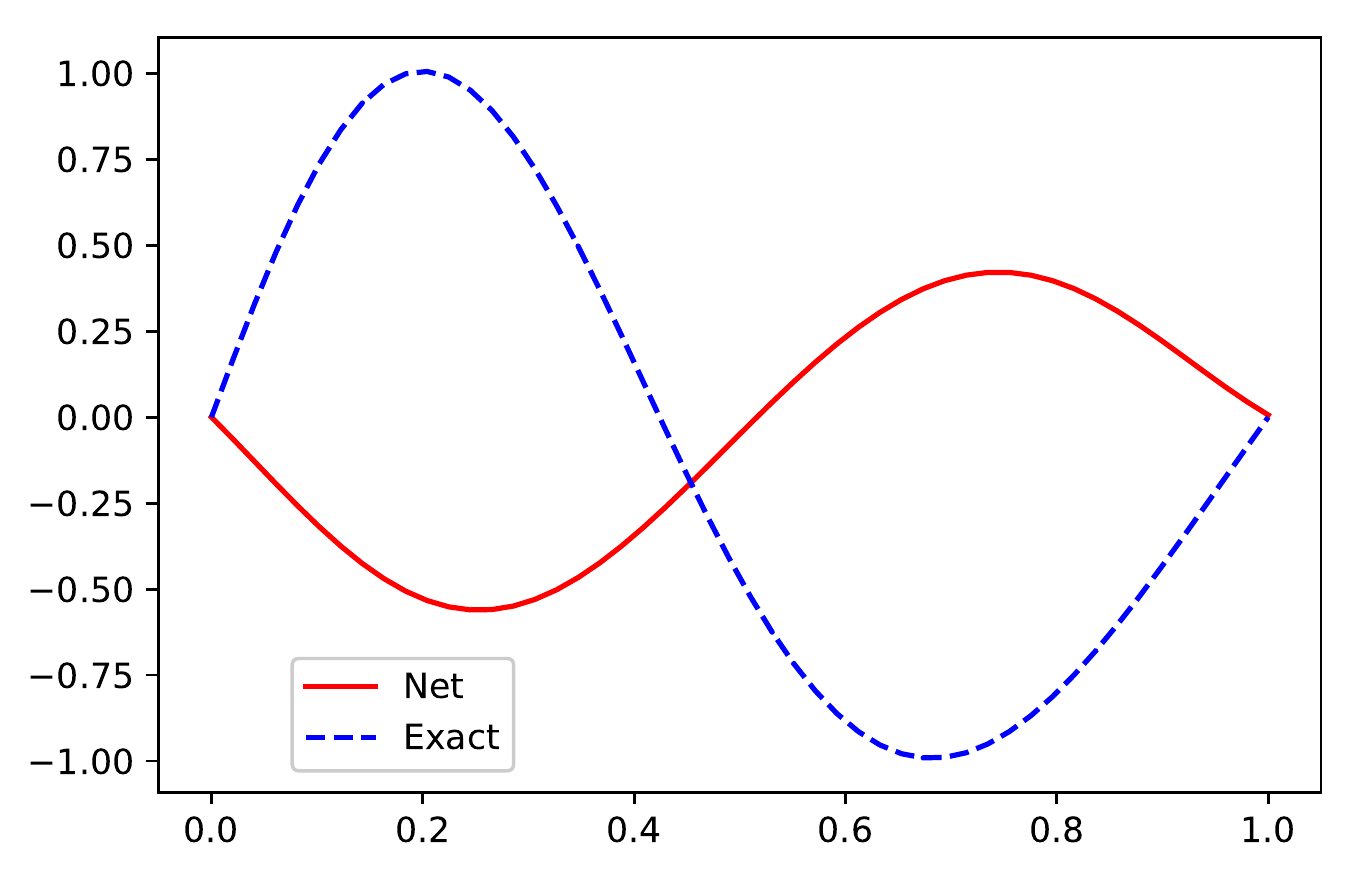}\label{fig:5-1}}
\hfil
\subfloat[]{\includegraphics[width=1.8in]{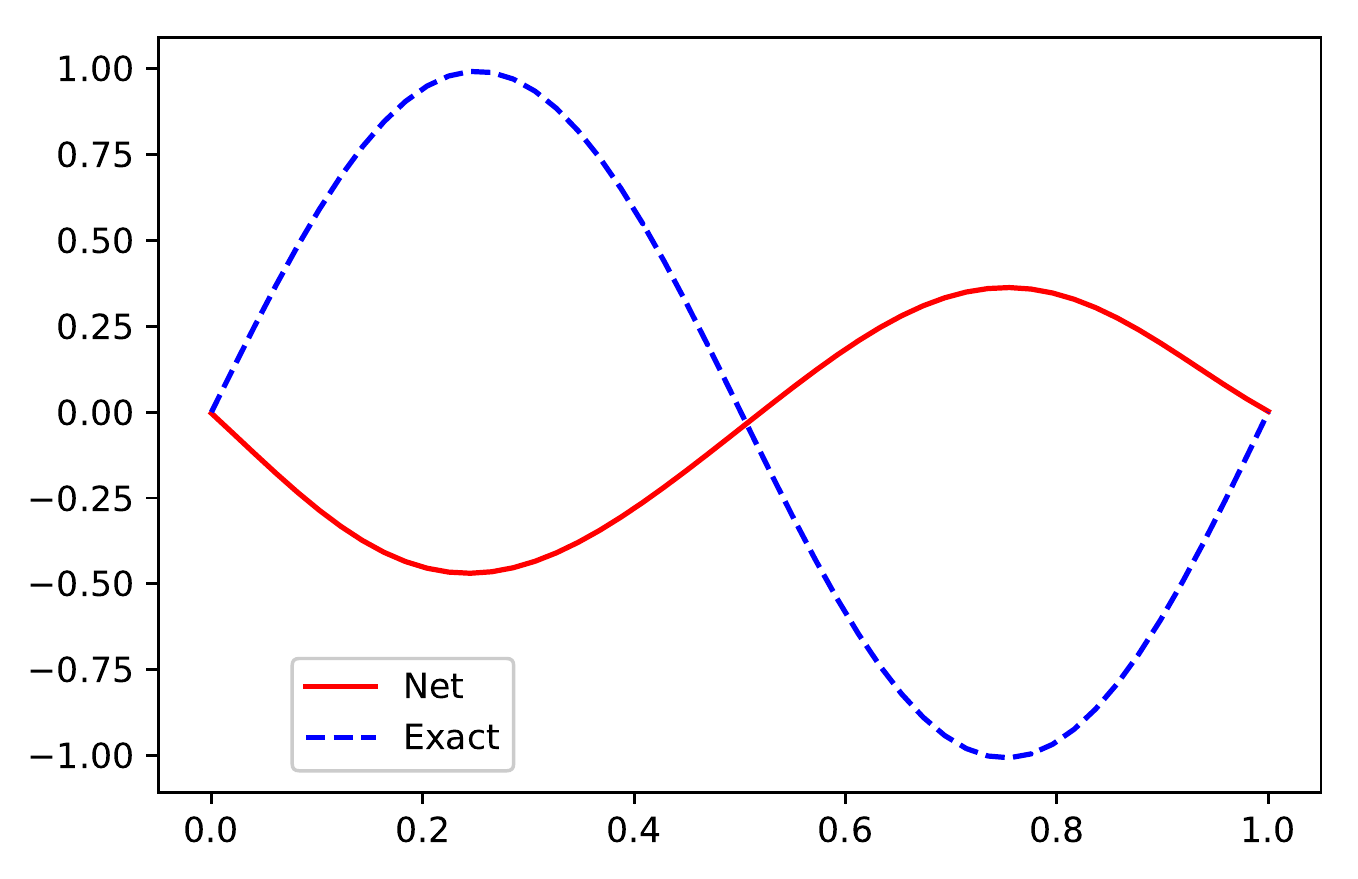}\label{fig:5-2}}
\hfil
\subfloat[]{\includegraphics[width=1.8in]{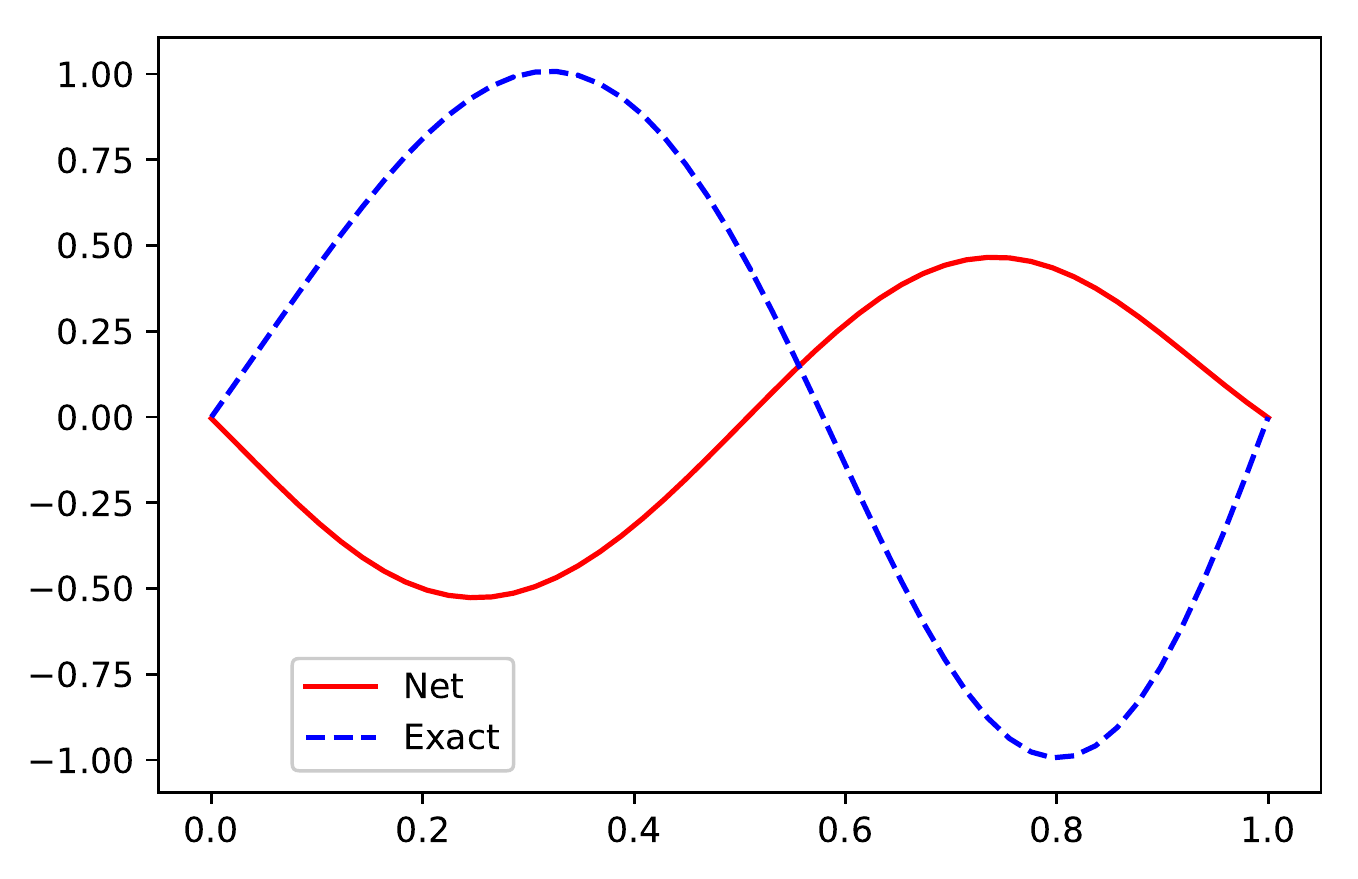}\label{fig:5-3}}
\vfil
\subfloat[]{\includegraphics[width=1.8in]{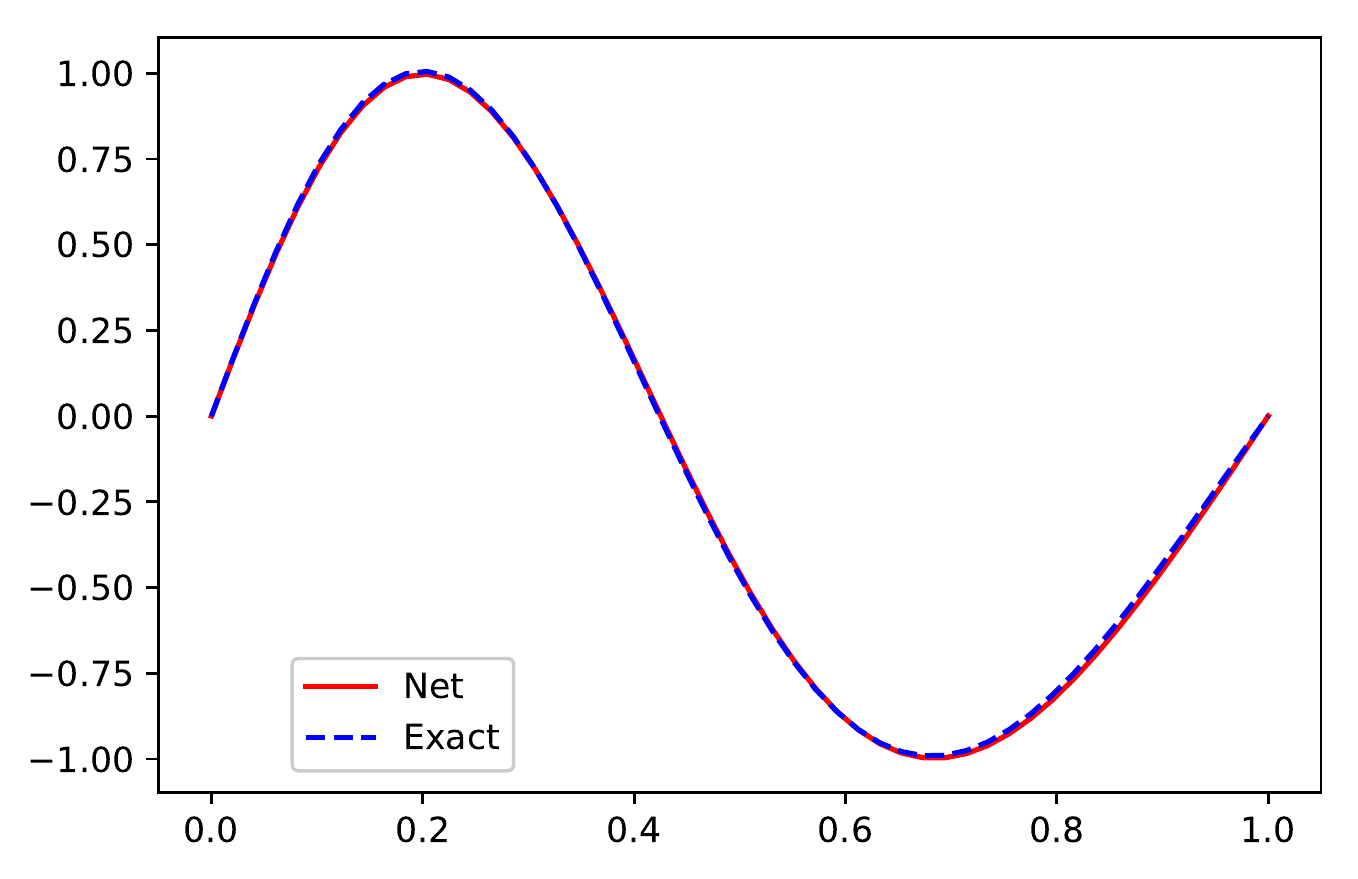}\label{fig:5-4}}
\hfil
\subfloat[]{\includegraphics[width=1.8in]{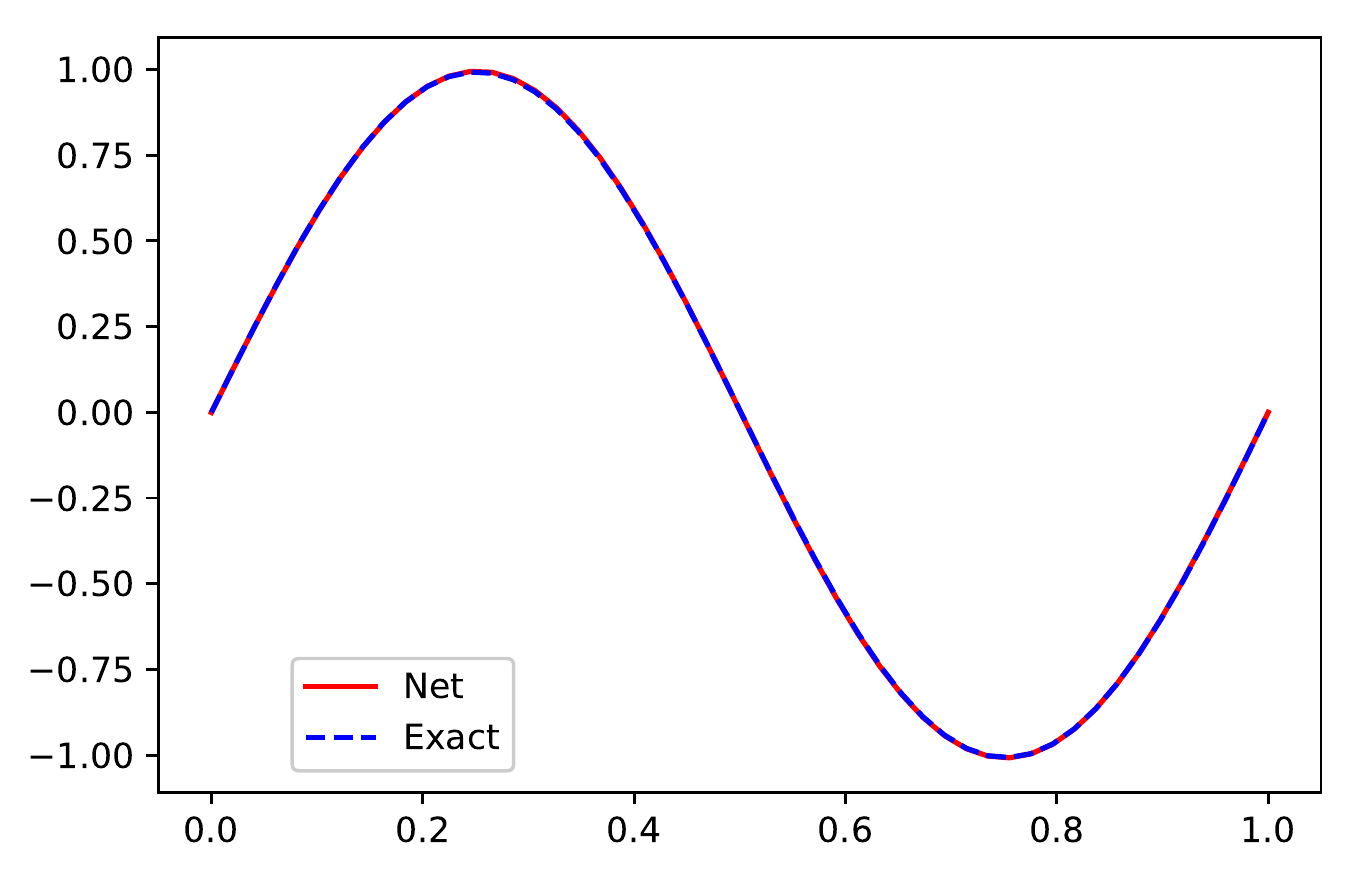}\label{fig:5-5}}
\hfil
\subfloat[]{\includegraphics[width=1.8in]{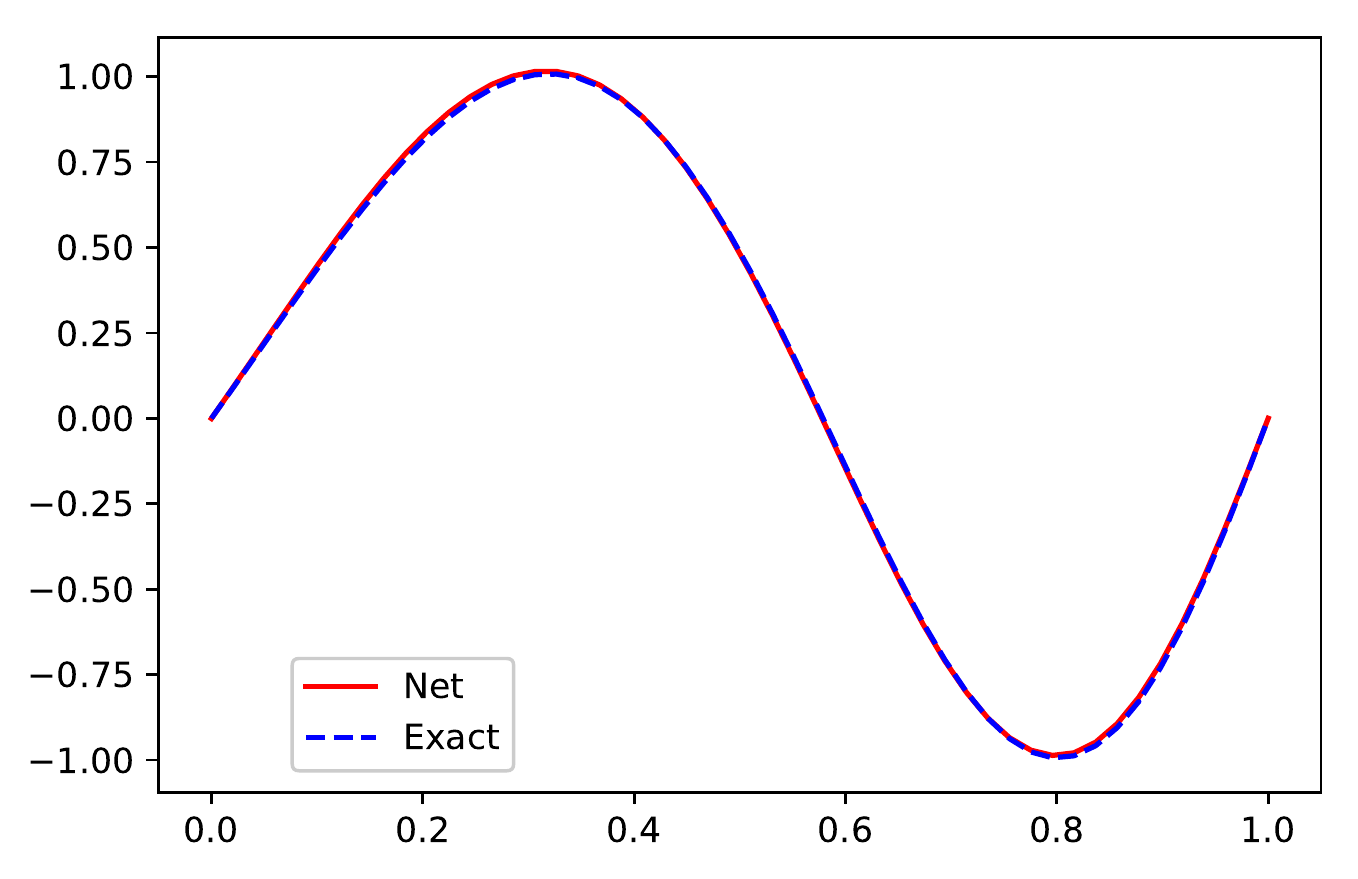}\label{fig:5-6}}
\vfil
\subfloat[]{\includegraphics[width=1.8in]{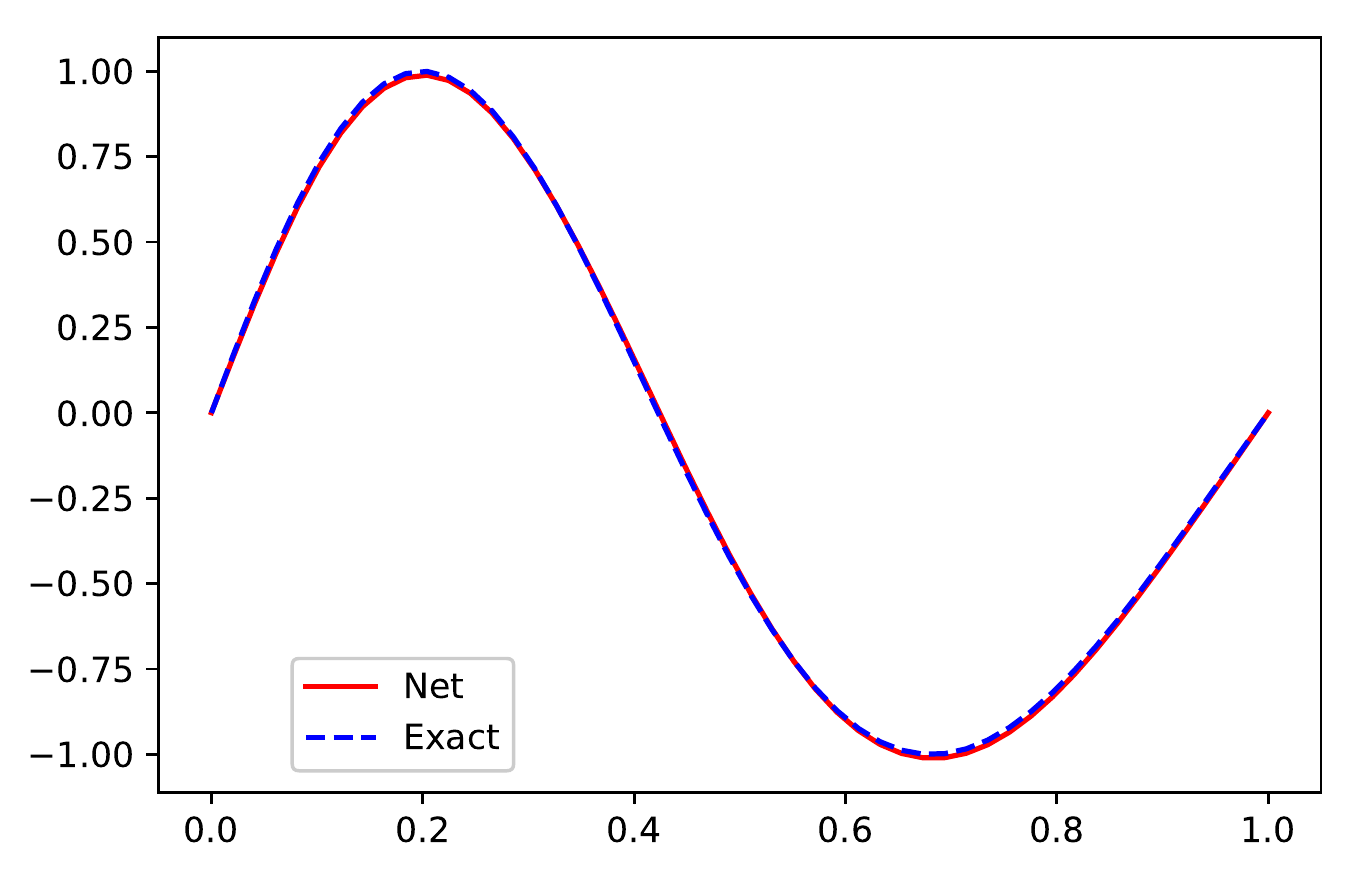}\label{fig:5-7}}
\hfil
\subfloat[]{\includegraphics[width=1.8in]{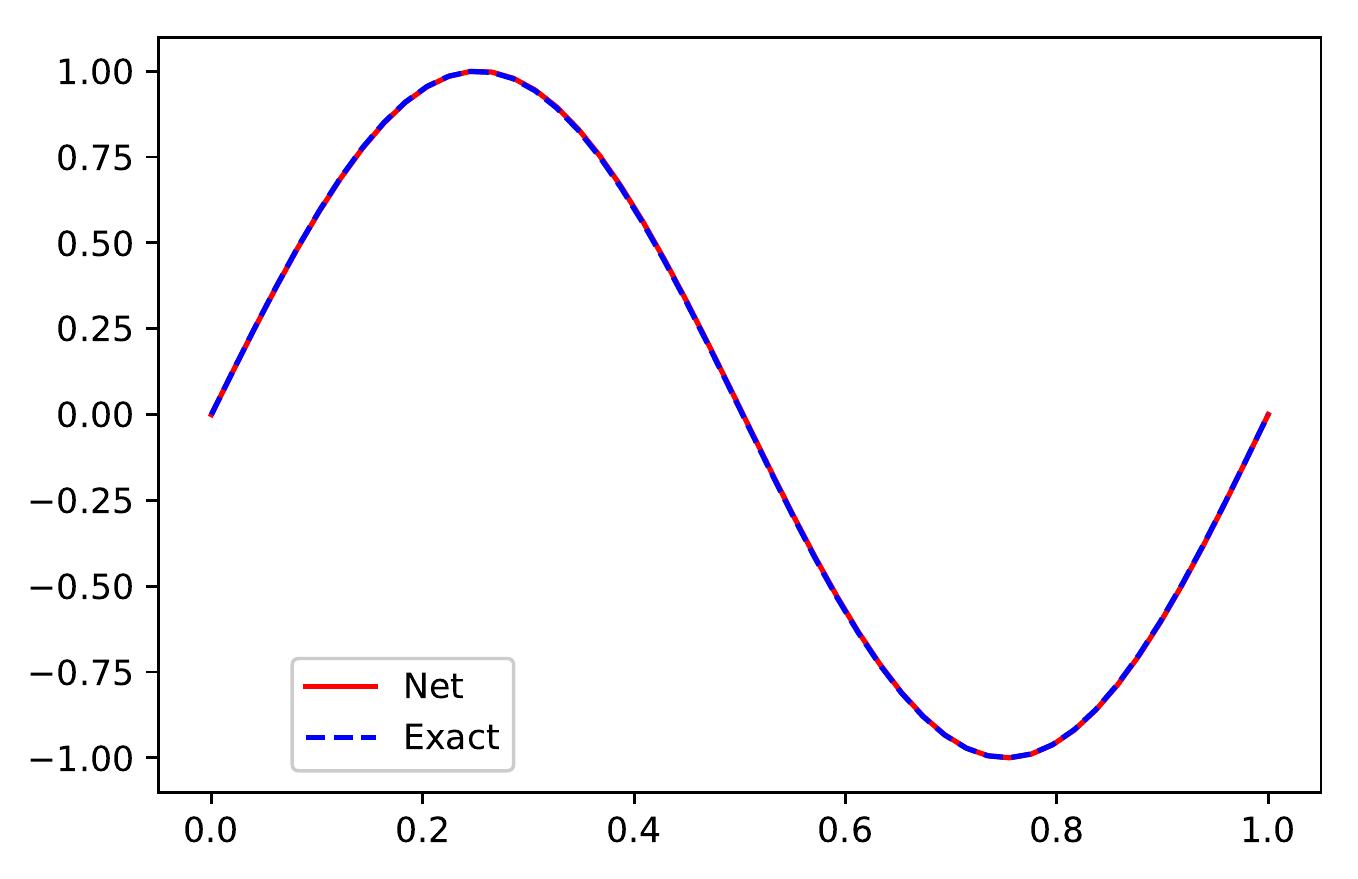}\label{fig:5-8}}
\hfil
\subfloat[]{\includegraphics[width=1.8in]{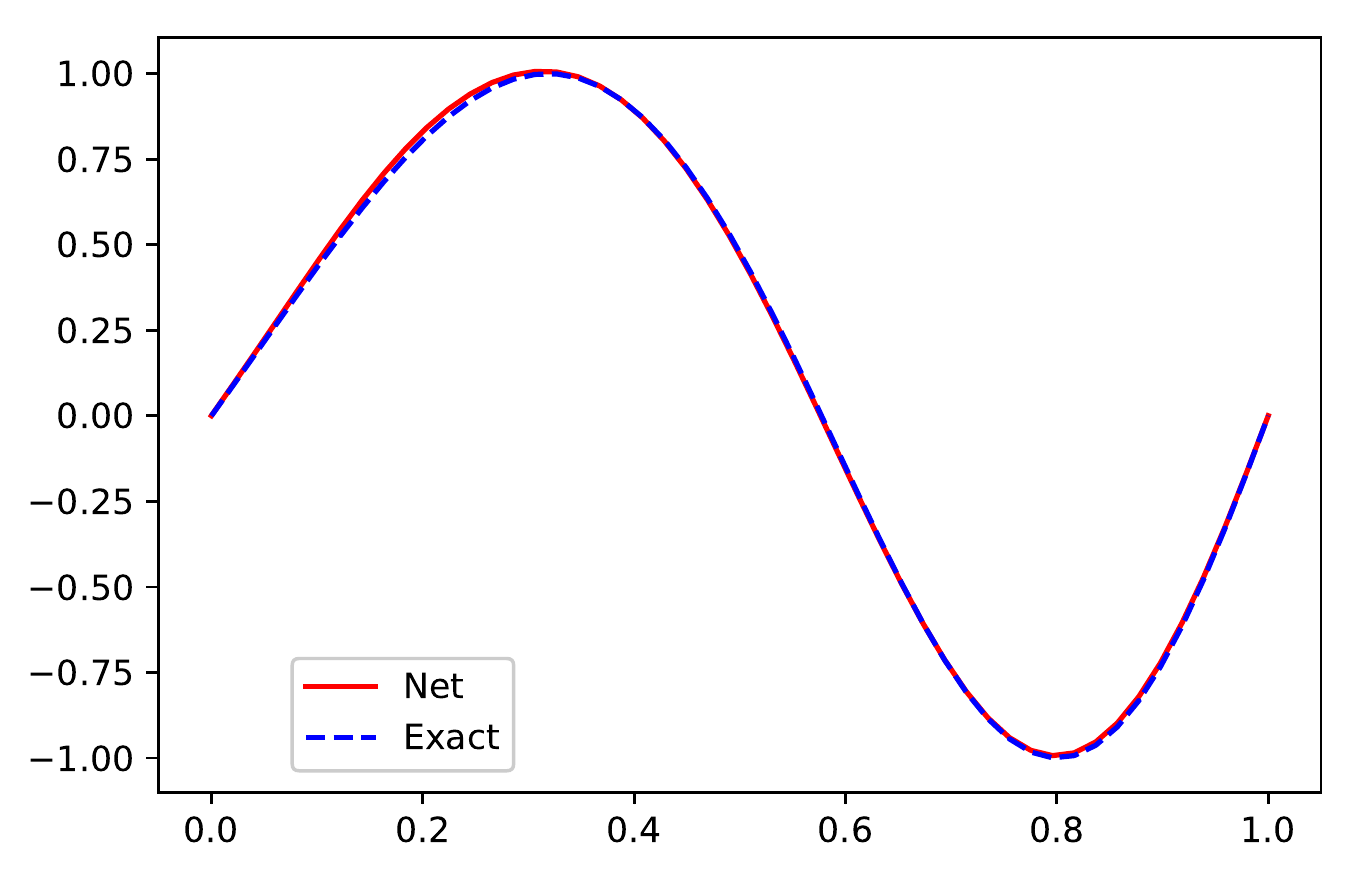}\label{fig:5-9}}
\caption{Comparison of numerical solution in 3D case obtained by non-AP-DNN and APFOS-DNN methods (\emph{Case 2}). From top to bottom: non-AP-DNN with $\varepsilon=1e-02$; APFOS-DNN with $\varepsilon=1e-02$; APFOS-DNN with $\varepsilon=1e-20$. From left to right: location $(x=0.5,z=0)$, $(x=0.5,z=0.5)$, and $(x=0.5,z=1)$.}
\label{fig:5}
\end{figure*}

\subsubsection{Data-driven discovery}

A reliable numerical prediction not only needs a good knowledge of physical model with initial/boundary conditions and an efficient numerical method, but the physical coefficient in the model. As mentioned in 2D and 3D forward modeling, the anisotropic strength $\varepsilon$ in \eqref{eq:anisotropic_pb} leads to solution with different characteristics, even makes the problem ill-posed. In this scenario, it is necessary to identify the coefficient $\varepsilon$ before forward modeling. Herein, we attempt to use the neural network-based approach to retrieve $\varepsilon$ by providing limited information of the solution, called observation in the field of inverse problem. The following architecture of neural network is established.

2D examples established by \emph{Setup I} with three cases of parameter settings are investigated, i.e., \emph{Case 1}, \emph{2} and \emph{3}. The identification problem is conducted at a square domain $\Omega=[0,1]\times[0,1]$ with $N\times N=100\times100$ grids. A neural network covering four hidden layers with $40$ and $60$ neurons per layer is adopted in the non-AP-DNN method and APFOS-DNN method. The numbers of the collocation points at the Neumann boundary $\Gamma_N=[0,1]\times\{0,1\}$ are set to $N_{n}=100$. The exact solution $\phi_{e}$ available at $N_{f}=3000$ points inside domain is taken as observation. Notice that the DNN-based identification problem includes additional training step in which only the model coefficient $\varepsilon$ is updated by the available observation (referring to Alg. \ref{alg:2}). The Adam algorithm is employed in the training step and set to $5000$ iterations. It is empirically observed that due to a reliable estimation on coefficient via such training step, the following simultaneous identification of coefficient $\varepsilon$ and optimization of the network parameters $\Theta$ could converge fast.

For identifying the anisotropic strength $\varepsilon$, empirically, in the loss function \eqref{eq:ident_discrete_LSF_AP2D} and \eqref{eq:ident_discrete_LSF_anisotropic} the weights in the term with respect to $\varepsilon$ should be changed larger for enhancing the importance of such terms. Herein, we take the weights $(\beta_{e}, \beta_{f_{1}}, \beta_{f_{2}}, \beta_{f_{3}}, \beta_{N_{1}}, \beta_{N_{2}})=(N_{f}, 1, 1, N_{f}, N_{n}, 1)$ and $(\alpha_{e}, \alpha_{f}, \alpha_{N})=(N_{f}, N_{f}, N_{n})$.

Table \ref{tab:ipcase123} displays the results of the identification problem obtained by DNN-based non-AP and APFOS methods, the $E_{2}$ errors and estimated coefficient $\varepsilon$. The promising results of APFOS-DNN approach indicate its capacity of dealing with the forward and inverse problem. In the contrary, the non-AP-DNN one could not yield results with desired accuracy and particularly it does not work in \emph{Case 3}. Finally, to show the convergence of the proposed method, the changing of estimated coefficient $\varepsilon$ with the increasing iterations is provided by Fig. \ref{fig:6-2}-\ref{fig:6-4}. It can be seen that the APFOS-DNN method converges fast and stably both in the training and prediction steps. Additionally, some test cases including observation with gaussian noise of different levels are taken into account. The results indicate that the APFOS-DNN method is robust to gaussian noise.

As a result, compared with the non-AP-DNN method, the proposed APFOS-DNN method is more suitable to identify any anisotropic strength $\varepsilon \geq 0$.

\begin{table}[]
\renewcommand{\arraystretch}{1.5}
    \caption{Comparison of identification results. Normal and bold fonts represent estimated $\varepsilon$ and $E_{2}$ errors of the numerical solution.}
    \vspace{10pt}
    \centering
    \begin{tabular}{c|c|ccc}
        \hline
        \diagbox{Scheme}{Exact $\varepsilon$} & Case & 1 & 1e-02 & 1e-20  \\
        \hline
        \multirow{6}*{\minitab[c]{Non-AP-DNN}}& \multirow{2}*{\minitab[c]{\emph{Case 1}}}  & 9.96e-01   & 9.97e-03             & 4.16e-05   \\
         && \textbf{3.24e-03}   & \textbf{1.22e-04}             & \textbf{1.40e-04}                       \\ \cline{3-5}
        &\multirow{2}*{\minitab[c]{\emph{Case 2}}}  & 9.97e-01   & 8.20e-03             & 1.54e-03                          \\
        && \textbf{3.55e-03}   & \textbf{1.70e-03}             & \textbf{1.51e-03}                      \\ \cline{3-5}
        &\multirow{2}*{\minitab[c]{\emph{Case 3}}}  & 4.38e-03   & 2.26e-02             & 2.17e-02                          \\
        && \textbf{3.58e-01}   & \textbf{2.23e-02}             & \textbf{2.26e-02}                      \\ \hline
        \multirow{6}*{\minitab[c]{APFOS-DNN}}&\multirow{2}*{\minitab[c]{\emph{Case 1}}}  & 1.00   & 1.00e-02            & 4.58e-07                  \\
        && \textbf{4.34e-04}   & \textbf{1.21e-04}            & \textbf{3.33e-05}                          \\ \cline{3-5}
        &\multirow{2}*{\minitab[c]{\emph{Case 2}}}   & 1.00            & 1.00e-02    & 1.42e-05                    \\
        && \textbf{4.99e-04}   & \textbf{8.29e-05}            & \textbf{2.73e-05}                          \\ \cline{3-5}
      &\multirow{2}*{\minitab[c]{\emph{Case 3}}}   & 9.71e-01            & 1.02e-02    & 4.11e-03  \\
       & & \textbf{7.36e-03}   & \textbf{2.25e-03}            & \textbf{1.93e-03}                     \\ \hline
    \end{tabular}
    \label{tab:ipcase123}
\end{table}

\begin{figure*}[!t]
\centering
\subfloat[Exact $\varepsilon=1$]{\includegraphics[width=2.5in]{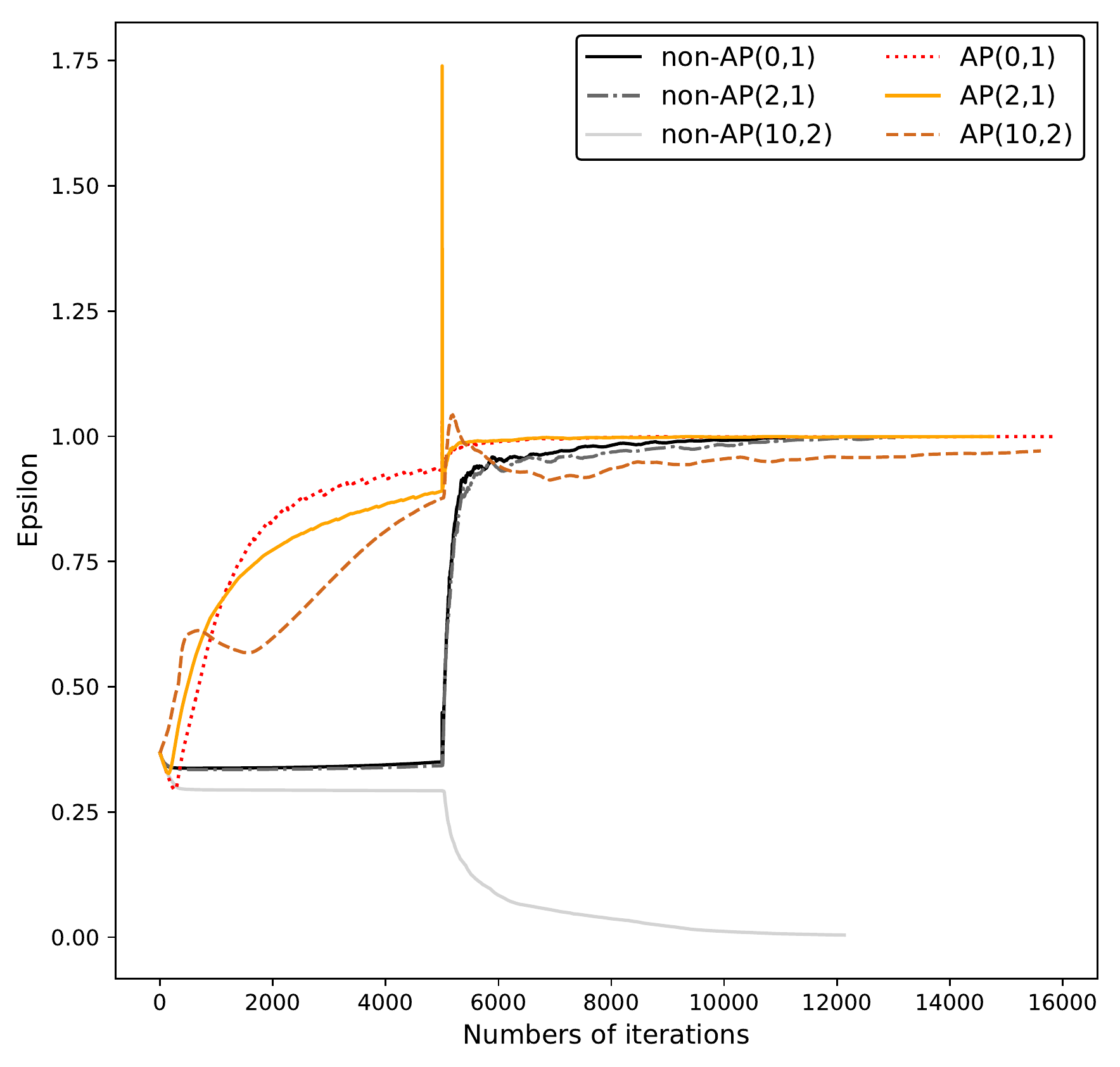}\label{fig:6-2}}
\hfil
\subfloat[Exact $\varepsilon=1e-02$]{\includegraphics[width=2.5in]{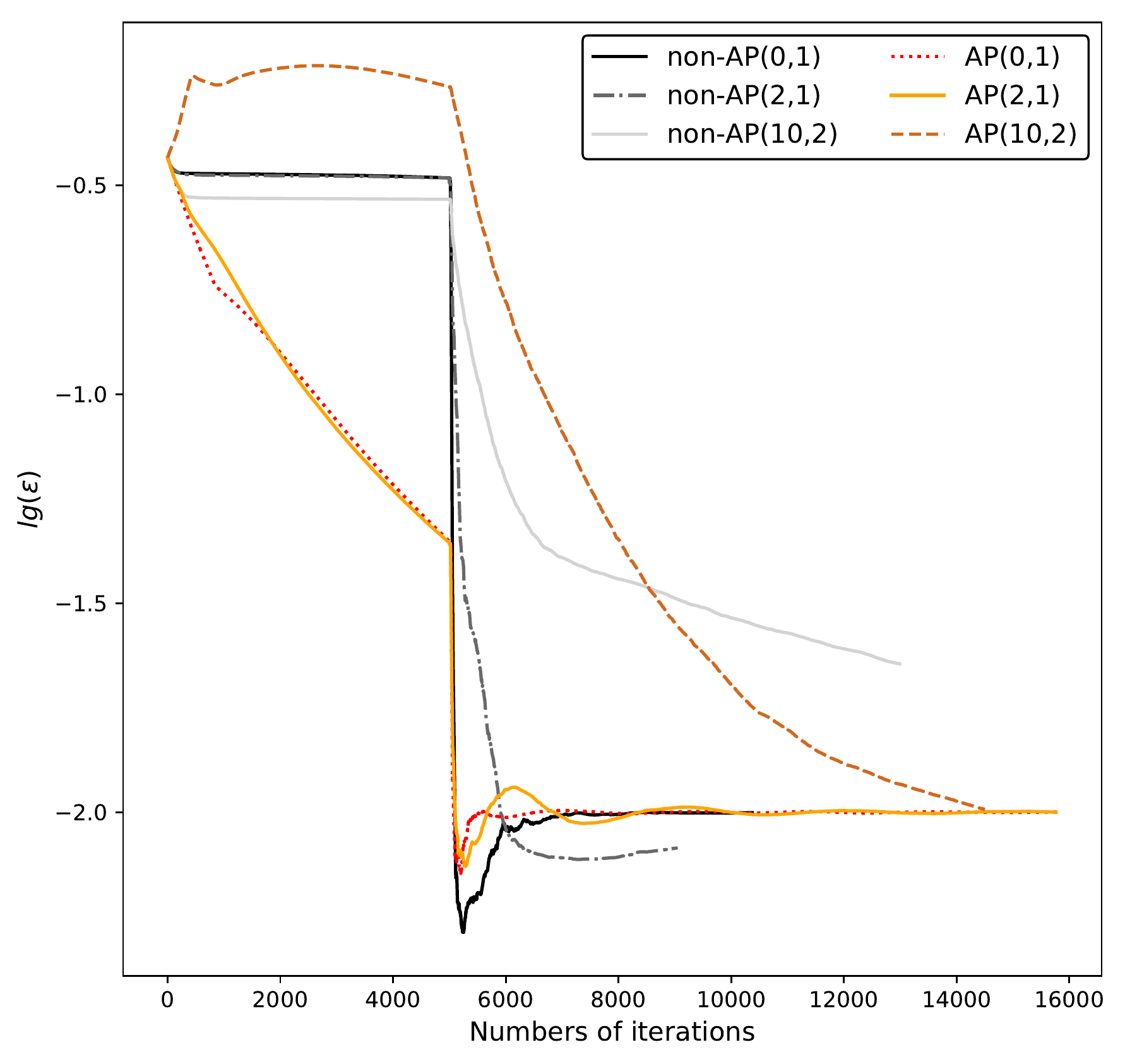}\label{fig:6-3}}
\vfil
\subfloat[Exact $\varepsilon=1e-20$]{\includegraphics[width=2.5in]{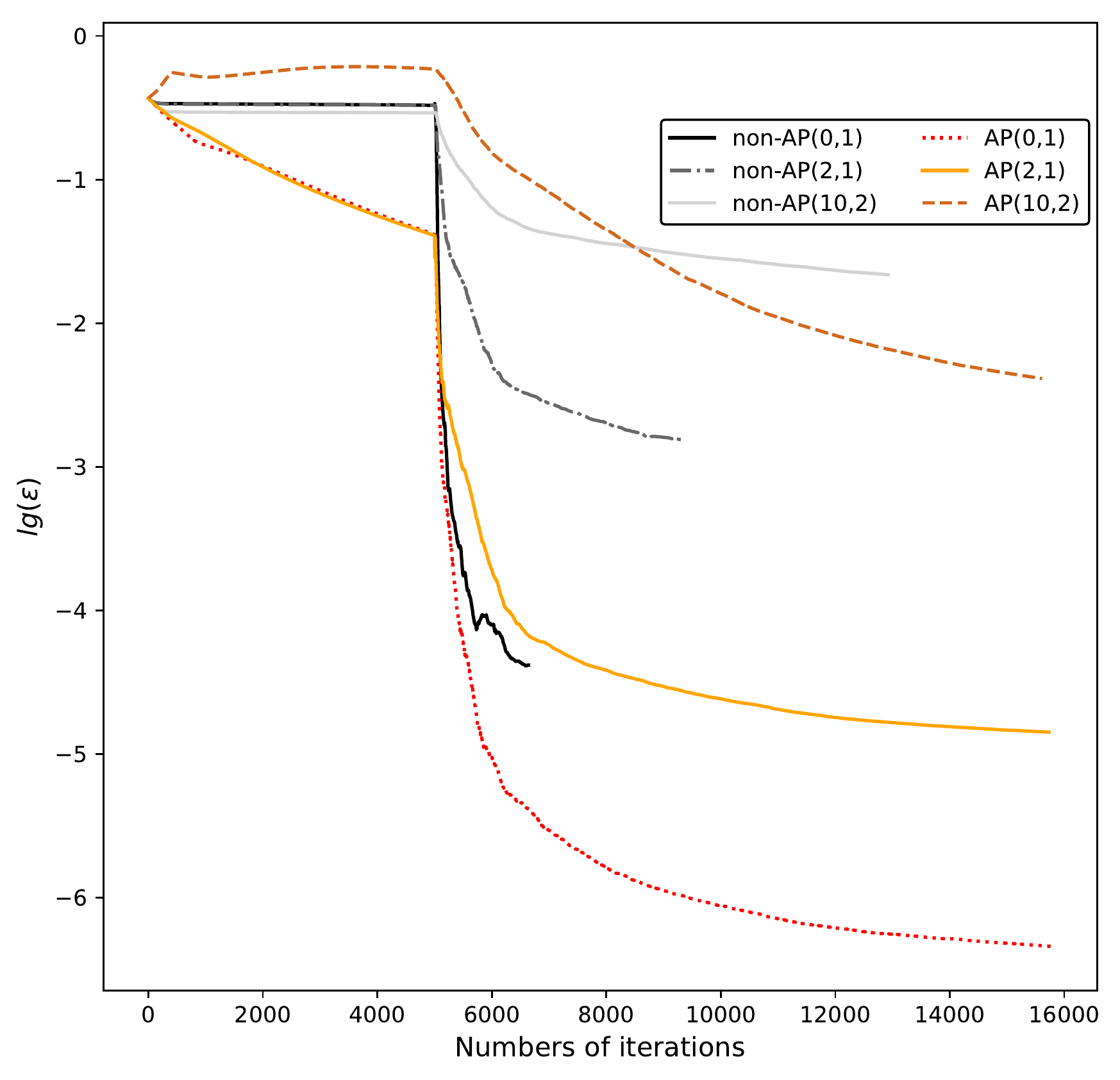}\label{fig:6-4}}
\caption{Comparison of identification results. Non-AP/AP$(\theta,m)$ in the legend represents the non-AP-DNN/APFOS-DNN method with parameters $(\theta,m)$.}
\label{fig:6}
\end{figure*}

\section{Conclusion and perspectives}

In this paper, an asymptotic preserving scheme based on first-order system (APFOS) for solving the anisotropic elliptic equations is proposed. The APFOS scheme is well-posed with respect to all anisotropic strength $\varepsilon\geq0$. We also rewrite the APFOS scheme into least-squares formulations. Then the APFOS scheme is numerically solved in deep neural network (DNN) framework. Comparing to non-AP method, our APFOS scheme is uniformly accurate with respect to all $\varepsilon$. It can treat various 2D/3D anisotropic elliptic equations with aligned, non-aligned, closed magnetic field. Finally, the APFOS scheme is applied to identify the anisotropic strength $\varepsilon$. Similarly, the APFOS scheme can predict accurately $\varepsilon$ in all cases, whereas the non-AP scheme does not work if we want to identify $\varepsilon\ll 1$.

The further coming work is to develop adaptive DNN for the APFOS scheme. Indeed, in current numerical performance of the APFOS scheme, we first select hyper parameters, such as number of layers, neurous, collocation points {\it  etc.}, from a simple twin experiment, then all tests will use those hyper parameters. However, due to the complexity of certain tests, such as 3D tests or closed magnetic field tests, the accuracy seems not to be optimal. As pointed in~\cite{cai2020deep}, an  adaptive DNN will be needed, where least-squares functional can be used as {\it a posteriori} estimator, thus the parameters of DNN will be selected in iterations.

\section{Acknowledgements}

This work has been supported by Heilongjiang Natural Science Foundation (LH2019A013).

\appendix
\section{The non-AP least squares schemes for the anisotropic equation~\eqref{eq:AP2Dnonaligned}}\label{sec:nonAP}

\subsection{The non-AP LS scheme for solving  the anisotropic equation~\eqref{eq:AP2Dnonaligned}}\label{sec:nonAPLS}

For the anisotropic elliptic equation~\eqref{eq:anisotropic_nonaligned2D},
the  LS formulations is to find $\phi\in H^2(\Omega)$ such that
\begin{align}\label{eq:continuous_LSFormulation_anisotropic}
\mathcal{L}(\phi;\mathbf{f}) = \min_{\psi\in H^2(\Omega)}\mathcal{L}(\psi;\mathbf{f}).
\end{align}
LS functional is inspired from the balanced least-squares (BLS) functional of~\cite{cai2020deep}, in which they point out that usual norm on boundary condition is weaker than that for the equation, thus a balanced LS functional  is adopted as follows
\begin{equation}\label{eq:continuous_LSF_anisotropic}
\mathcal{L}(\psi;\mathbf{f}) = \|\varepsilon\Delta_\bot\psi + \Delta_\|\psi + \varepsilon f\|^2_{0,\Omega} + \|\psi-g\|^2_{3/2,\Gamma_D} + \|\varepsilon\nabla_\bot \psi \cdot \mathbf{n} +
\nabla_{\|}\psi \cdot \mathbf{n} \|^2_{1/2,\Gamma_N}.
\end{equation}

Then we discretize~\eqref{eq:continuous_LSFormulation_anisotropic}-\eqref{eq:continuous_LSF_anisotropic} as follows. The discrete LS formulation reads
\begin{align*}
\hat{\mathcal{L}}[\hat{\phi};\mathbf{f}]({\Theta}) = \min_{\tilde{\Theta}\in\mathbb{R}^N}\hat{\mathcal{L}}[\hat{\psi} ;\mathbf{f}](\tilde{\Theta}),
\end{align*}
where $\hat{\phi} = \hat{\phi}(\mathbf{x},\Theta)$ represents an approximation of solution of anisotropic equation~\eqref{eq:anisotropic_nonaligned2D}, depending on the location $\mathbf{x}\in\mathbb{R}^2$ and  the  weights and biases  denoted as $\Theta\in\mathbb{R}^N$.
The corresponding discrete LS functional is
\begin{equation}\label{eq:discrete_LSF_anisotropic}
\begin{array}{ll}
\hat{\mathcal{L}}[\hat{\psi} ;\mathbf{f}](\tilde{\Theta})
= &\frac{1}{N_f}\sum\limits_{i=1}^{N_f}\left(\varepsilon\Delta_\bot\hat{\psi}(\mathbf{x}_i,\tilde{\Theta}) + \Delta_\|\hat{\psi}(\mathbf{x}_i,\tilde{\Theta}) + \varepsilon f(\mathbf{x}_i)\right)^2 +\frac{\alpha_D}{N_d}\sum\limits_{j=1}^{N_d}\left(\hat{\psi}(\mathbf{x}_j,\tilde{\Theta})- g(\mathbf{x}_j)\right)^2\\[3mm]
&  +\frac{\alpha_N}{N_n}\sum\limits_{k=1}^{N_n}\left(\varepsilon\nabla_\bot \hat{\psi}(\mathbf{x}_k,\tilde{\Theta}) \cdot \mathbf{n}(\mathbf{x}_k) +
\nabla_{\|}\hat{\psi}(\mathbf{x}_k,\tilde{\Theta}) \cdot \mathbf{n}(\mathbf{x}_k)\right)^2.
\end{array}
\end{equation}
Notice that the discrete LS functional~\eqref{eq:discrete_LSF_anisotropic} is an approximation of~\eqref{eq:continuous_LSF_anisotropic}, where $H^{3/2}$ and $H^{1/2}$ norms are replaced by weighted $L^2$ norms. The weights $\alpha_D$ and $\alpha_N$ are specified in Section~\ref{sec:numresults}.

\subsection{The non-AP LS scheme for identifying the anisotropic strength $\varepsilon$}
\label{sec:nonAPLS_discrete}
Let us now consider a non-AP identification scheme for the anisotropic problem~\eqref{eq:anisotropic_nonaligned2D}.
The identification LS formulations is to find $\varepsilon\in\mathbb{R}^+$ and $\phi\in H^2(\Omega)$ such that
\begin{align*}
\mathcal{L}(\varepsilon,\phi;\mathbf{f}) = \min_{\tilde{\varepsilon}\in\mathbb{R}^+,\psi\in H^2(\Omega)}\mathcal{L}(\tilde{\varepsilon},\psi;\mathbf{f}).
\end{align*}
The corresponding LS functional, by adding \emph{priori} information $\phi_e$ to~\eqref{eq:continuous_LSF_anisotropic}, reads
\begin{equation}\label{eq:iden_nonap-LSfct}
\begin{array}{ll}
\mathcal{L}(\tilde{\varepsilon},\psi;\mathbf{f}) = & \|\psi - \phi_e\|^2_{0,\Omega} \\[3mm]
& + \|\tilde{\varepsilon}\Delta_\bot\psi + \Delta_\|\psi + \tilde{\varepsilon} f\|^2_{0,\Omega} + \|\psi-g\|^2_{3/2,\Gamma_D} + \|\tilde{\varepsilon}\nabla_\bot \psi \cdot \mathbf{n} +
\nabla_{\|}\psi \cdot \mathbf{n} \|^2_{1/2,\Gamma_N},
\end{array}
\end{equation}
where $\mathbf{f}= (f,g)$, $\tilde{\varepsilon}\in\mathbb{R}^+$ and $\psi\in H^2(\Omega)$.

The discrete LS formulation reads
\begin{subequations}\label{eqs:ident_discrete_LSFormulation_anisotropic}
\begin{align}\label{eq:ident_discrete_LSFormulation_anisotropic}
\hat{\mathcal{L}}[\hat{\phi};\mathbf{f}](\varepsilon^*,\Theta) = \min_{\tilde{\varepsilon}^*\in\mathbb{R},\tilde{\Theta}\in\mathbb{R}^N}\hat{\mathcal{L}}[\hat{\psi} ;\mathbf{f}](\tilde{\varepsilon}^*,\tilde{\Theta}).
\end{align}
where the corresponding discrete  LS functional to~\eqref{eq:iden_nonap-LSfct} is
\begin{equation}\label{eq:ident_discrete_LSF_anisotropic}
\begin{array}{ll}
\hat{\mathcal{L}}[\hat{\psi} ;\mathbf{f}](\tilde{\varepsilon}^*,\tilde{\Theta})
= & \frac{\alpha_{e}}{N_f}\sum\limits_{i=1}^{N_f}\left(\hat{\psi}(\mathbf{x}_i,\tilde{\Theta}) - \phi_e(\mathbf{x}_i)\right)^2 \\[3mm]
%
%
 &+ \frac{\alpha_{f}}{N_f}\sum\limits_{i=1}^{N_f}\left(\exp(\tilde{\varepsilon}^*)\Delta_\bot\hat{\psi}(\mathbf{x}_i,\tilde{\Theta}) + \Delta_\|\hat{\psi}(\mathbf{x}_i,\tilde{\Theta}) + \exp(\tilde{\varepsilon}^*) f(\mathbf{x}_i)\right)^2 \\[3mm]
&  +\frac{\alpha_N}{N_n}\sum\limits_{k=1}^{N_n}\left(\exp(\tilde{\varepsilon}^*)\nabla_\bot \hat{\psi}(\mathbf{x}_k,\tilde{\Theta}) \cdot \mathbf{n}(\mathbf{x}_k) +
\nabla_{\|}\hat{\psi}(\mathbf{x}_k,\tilde{\Theta}) \cdot \mathbf{n}(\mathbf{x}_k)\right)^2.
\end{array}
\end{equation}
\end{subequations}

\end{document}